\documentclass[english]{article}
\usepackage[T1]{fontenc}
\usepackage[utf8]{inputenc}
\usepackage{lmodern}
\usepackage{amssymb,authblk}
\usepackage{mathtools}
\usepackage{amsthm}
\usepackage{stmaryrd}
\usepackage{mathrsfs}
\usepackage{amsfonts}
\usepackage{tdsfrmath}
\usepackage{faktor}
\usepackage{xcolor}
\usepackage{tcolorbox}
\usepackage[all,cmtip, color,matrix,arrow]{xy} 
\usepackage[linewidth=1pt,
middlelinecolor= black,
middlelinewidth=0.4pt,
roundcorner=1pt,
topline = false,
rightline = false,
bottomline = false,
rightmargin=0pt,
skipabove=0pt,
skipbelow=0pt,
leftmargin=-1cm,
innerleftmargin=1cm,
innerrightmargin=0pt,
innertopmargin=0pt,
innerbottommargin=0pt]{mdframed}

\usepackage{array}
\usepackage{nicematrix}
\usepackage[a4paper]{geometry}
\usepackage{shuffle}
\usepackage[english]{babel}
\usepackage{subcaption}
\usepackage[plainpages=false,pdfpagelabels,pagebackref]{hyperref} 
\usepackage{tikz-cd}
\usepackage{fp}
\usepackage{ifthen}
\usepackage{calc}
\usetikzlibrary{automata, positioning, arrows}

\newtheorem{innercustomgeneric}{\customgenericname}
\providecommand{\customgenericname}{}
\newcommand{\newcustomtheorem}[2]{%
	\newenvironment{#1}[1]
	{%
		\renewcommand\customgenericname{#2}%
		\renewcommand\theinnercustomgeneric{##1}%
		\innercustomgeneric
	}
	{\endinnercustomgeneric}
}
\newcustomtheorem{customdefi}{Definition}
\newcustomtheorem{customrq}{Remark}
\newcustomtheorem{customEgs}{Examples}
\newcustomtheorem{customEg}{Example}

\newcounter{thmcount}
\theoremstyle{plain}
\newtheorem{thm}[thmcount]{Theorem}
\newtheorem{Prop}[thmcount]{Proposition}
\newtheorem{Cor}[thmcount]{Corollary}
\newtheorem{Lemme}[thmcount]{Lemma}
\newtheorem{conj}[thmcount]{Conjecture}

\theoremstyle{definition}
\newtheorem{Eg}[thmcount]{\textbf{Example}}
\newtheorem{Egs}[thmcount]{\textbf{Examples}}
\newtheorem{defi}[thmcount]{Definition}
\newtheorem{Not}[thmcount]{Notation}
\newtheorem{Rq}[thmcount]{Remark}

\newtheorem{innercustomgenerictwo}{\customgenericname}
\providecommand{\customgenericname}{}
\newcommand{\newcustomtheoremplain}[2]{%
	\newenvironment{#1}[1]
	{%
		\renewcommand\customgenericname{#2}%
		\renewcommand\theinnercustomgenerictwo{##1}%
		\innercustomgenerictwo
	}
	{\endinnercustomgeneric}
}
\newcustomtheoremplain{customthm}{Theorem}
\newcustomtheoremplain{customlemma}{Lemma}
\newcustomtheoremplain{customprop}{Proposition}

\newcommand{\IEM}[2]{\llbracket #1,#2 \rrbracket}

\newcommand{\K}{\ensuremath{\mathbb{K}}}

\newcommand{\qsh}{quasi-shuffle}

\newcommand{\zetash}{\zeta_\shuffle}
\newcommand{\zetaTsh}{\zeta^T_\shuffle}
\newcommand{\zetaT}{\zeta^T}

\NewDocumentCommand{\MZVs}{ }{Multiple Zeta Values}

\NewDocumentCommand{\AZVs}{}{Arborified Zeta Values}

\DeclareMathOperator{\QSh}{Qsh}
\DeclareMathOperator{\Sh}{Sh}

\DeclareMathOperator{\dt}{dt}

\DeclareMathOperator{\ima}{Im}

\DeclareMathOperator{\flaten}{flat}
\DeclareMathOperator{\Flaten}{Flat}

\DeclareMathOperator{\B}{B}

\DeclareMathOperator{\Ker}{Ker}
\DeclareMathOperator{\lin}{lin}
\DeclareMathOperator{\Supp}{\mathrm{Supp}}

\newcommand{\calT}{\mathcal{T}}

\newcommand{\calW}{\mathcal{W}}

\newcommand{\W}{\mathcal{W}_{\N^*}}
\newcommand{\Wcv}{\mathcal{W}^{\mathrm{conv}}_{\N^*}}
\newcommand{\Wxy}{\mathcal{W}_{\{x,y\}}}
\newcommand{\Wxycv}{\mathcal{W}^{\mathrm{conv}}_{\{x,y\}}}
\newcommand{\WO}{\mathcal{W}_{\Omega}}

\newcommand{\PolW}{\Q\lbrack{\rm W}^{\rm conv}_{\N^*}\rbrack}
\newcommand{\PolT}{\Q\lbrack{\rm PT}^{\rm conv}_{\N^*}\rbrack}

\newcommand{\fraks}{\mathfrak{s}} 
\newcommand{\fraksPT}{\mathfrak{s}^{PT}} 

\newcommand{\PTN}{\text{PT}_{\N^*}}
\newcommand{\PT}{\text{PT}}
\newcommand{\PTO}{\text{PT}_{\Omega}}

\newcommand{\calPT}{\mathcal{PT}}
\newcommand{\calPTO}{\mathcal{PT}_\Omega}
\newcommand{\calPTN}{\mathcal{PT}_{\N^*}}
\newcommand{\calPTNconv}{\mathcal{PT}_{\N^*}^{\rm conv}}
\newcommand{\calPTxy}{\mathcal{PT}_{\{x,y\}}}
\newcommand{\calPTxyconv}{\mathcal{PT}_{\{x,y\}}^{\rm conv}}

\newcommand{\pr}{\operatorname{pr}}

\newcommand{\tdun}[1]{\begin{picture}(10,5)(-2,-1)
\put(0,0){\circle*{2}}
\put(2,-2){$\tiny #1$}
\end{picture}}

\newcommand{\tddeux}[2]{\begin{tikzpicture}[line cap=round,line join=round,>=triangle 45,x=0.2cm,y=0.2cm, anchor=base, baseline]
		\draw (0,0) node[right] {\footnotesize{$#1$}};
		\draw (0,0) -- (0,1);
		\draw (0,1) node[right] {\footnotesize{$#2$}};
		\filldraw[color=black] (0,0) circle (0.15);
		\filldraw[color=black] (0,1) circle (0.15);
\end{tikzpicture}}

\NewDocumentCommand{\labpoint}{m}{\begin{tikzpicture}[anchor =base, baseline, line cap=round,line join=round,>=triangle 45,x=0.2cm,y=0.2cm]
		\filldraw[color=black] (0,0) circle (0.15);
		\draw (0,0) node[right] {\footnotesize{$#1$}};		
\end{tikzpicture}}

\NewDocumentCommand{\labY}{m m m}{\raisebox{-0.4\height}{\begin{tikzpicture}[line cap=round,line join=round,>=triangle 45,x=0.4cm,y=0.4cm]
			\draw [line width=.5pt] (0.,1.)-- (-1.,2.);
			\draw [line width=.5pt] (0.,1.)-- (1.,2.);
			\draw[above] (0,1) node {\footnotesize{#1}};
			\draw[above] (1,2) node {\footnotesize{#3}};
			\draw[above] (-1,2) node {\footnotesize{#2}};
			\filldraw (0,1) circle (1pt);
			\filldraw (1,2) circle (1pt);
			\filldraw (-1,2) circle (1pt);
\end{tikzpicture}}}

\NewDocumentCommand{\labYr}{m m m m}{\raisebox{-0.4\height}{\begin{tikzpicture}[line cap=round,line join=round,>=triangle 45,x=0.4cm,y=0.4cm]
			\draw [line width=.5pt] (0.,1.)-- (-1.,2.);
			\draw [line width=.5pt] (0.,1.)-- (1.,2.);
			\draw [line width=.5pt] (1.,2.)-- (1.,3.);
			\draw[above] (0,1) node {\footnotesize{#1}};
			\draw[above] (-1,2) node {\footnotesize{#2}};
			\draw[left] (1,2) node {\footnotesize{#3}};
			\draw[above] (1,3) node {\footnotesize{#4}};
			\filldraw (0,1) circle (1pt);
			\filldraw (1,2) circle (1pt);
			\filldraw (1,3) circle (1pt);
			\filldraw (-1,2) circle (1pt);
\end{tikzpicture}}}

\NewDocumentCommand{\labYl}{m m m m}{\raisebox{-0.4\height}{\begin{tikzpicture}[line cap=round,line join=round,>=triangle 45,x=0.4cm,y=0.4cm]
			\draw [line width=.5pt] (0.,1.)-- (-1.,2.);
			\draw [line width=.5pt] (0.,1.)-- (1.,2.);
			\draw (-1,2) -- (-1,3);
			\draw[above] (0,1) node {\footnotesize{#1}};
			\draw[right] (-1,2) node {\footnotesize{#2}};
			\draw[above] (-1,3) node {\footnotesize{#3}};
			\draw[right] (1,2) node {\footnotesize{#4}};
			\filldraw (0,1) circle (1pt);
			\filldraw (1,2) circle (1pt);
			\filldraw (-1,3) circle (1pt);
			\filldraw (-1,2) circle (1pt);
\end{tikzpicture}}}

\NewDocumentCommand{\labBY}{m m m m}{\raisebox{-0.4\height}{\begin{tikzpicture}[line cap=round,line join=round,>=triangle 45,x=0.4cm,y=0.4cm]
			\draw [line width=.5pt] (0.,0)-- (0,1);
			\draw [line width=.5pt] (0.,1.)-- (-1.,2.);
			\draw [line width=.5pt] (0.,1.)-- (1.,2.);
			\filldraw (0,0);
			\draw[right] (0,0) node {\footnotesize{#1}};
			\draw[above] (0,1) node {\footnotesize{#2}};
			\draw[above] (-1,2) node {\footnotesize{#3}};
			\draw[above] (1,2) node {\footnotesize{#4}};
			\filldraw (0,0) circle (1pt);
			\filldraw (0,1) circle (1pt);
			\filldraw (1,2) circle (1pt);
			\filldraw (-1,2) circle (1pt);
\end{tikzpicture}}}

\newcommand{\decombright}[7]{\raisebox{-0.5\height}{\begin{tikzpicture}[line cap=round,line join=round,>=triangle 45,x=0.3cm,y=0.3cm]
			\draw (0,1)--(-1,2) node[left]{$#1$};
			\draw (0,1)--(6,7);
			\draw (2,3)--(1,4) node[left]{$#2$};
			\draw (3.,4.5) node[rotate=45]{$\cdots$};
			\draw (5,6)--(4,7) node[left]{$#3$};
			\draw (0,1) node[right]{\scriptsize{$#4$}};
			\draw (2,3) node[right]{\scriptsize{$#5$}};
			\draw (5,6) node[right]{\scriptsize{$#6$}};
			\draw (6,7) node[right]{\scriptsize{$#7$}};
			\filldraw (0,1) circle (0.15);
			\filldraw (2,3) circle (0.15);
			\filldraw (5,6) circle (0.15);
			\filldraw (6,7) circle (0.15);
\end{tikzpicture} }}

\newcommand{\decombleft}[7]{\raisebox{-0.5\height}{\begin{tikzpicture}[line cap=round,line join=round,>=triangle 45,x=0.3cm,y=0.3cm]
			\draw (0,1)--(1,2) node[right]{$#1$};
			\draw (0,1)--(-6,7);
			\draw (-2,3)--(-1,4) node[right]{$#2$};
			\draw (-3.,4.5) node[rotate=135]{$\cdots$};
			\draw (-5,6)--(-4,7) node[right]{$#3$};
			\draw (0,1) node[left]{\scriptsize{$#4$}};
			\draw (-2,3) node[left]{\scriptsize{$#5$}};
			\draw (-5,6) node[left]{\scriptsize{$#6$}};
			\draw (-6,7) node[left]{\scriptsize{$#7$}};
			\filldraw (0,1) circle (0.15);
			\filldraw (-2,3) circle (0.15);
			\filldraw (-5,6) circle (0.15);
			\filldraw (-6,7) circle (0.15);
\end{tikzpicture}}}

	\title{Regularised double shuffle relations for planar \AZVs{}}
	
	\author{Pierre Catoire${}^{1}$, Pierre~J.~Clavier${}^{2,3}$, Ku-Yu Fan${}^{4}$\\
~\\
\normalsize \it $^1$ \small{Université de Montpellier, IMAG, Institut Montpelliérain Alexander Grothendieck,
				Place Eugène Bataillon, Montpellier, 34070, France.} \\
\normalsize \it $^2$  Department of Mathematics, IRIMAS, Université de Haute Alsace.\\
\normalsize \it $^3$ IRMA, Université de Strasbourg.\\
\normalsize \it $^4$ Graduate School of Mathematics, Nagoya University, Furo-cho, Chikusa-ku, Nagoya, 464-8602, Japan.\\
~\\
\normalsize email: pierre.clavier@uha.fr, catoire\_research@proton.me, ku-yu.fan.d2@math.nagoya-u.ac.jp}

\date{}

\begin{document}	
	\maketitle
	
	\begin{abstract} 
We endow spaces of decorated planar rooted trees with new dendriform and tridendrifrom algebra structures and provide their combinatorial description. We then show that the planar counterparts of \AZVs{} are algebra morphisms for these shuffle and quasi-shuffle products of planar rooted trees. We also prove an arborified version of Hoffman's regularisation relation for \AZVs{}. We conjecture that those give every rational relation between \AZVs{} and show that this conjecture implies the regularised double shuffle conjecture for \MZVs{}.
\end{abstract}
{\bf Keywords:} Dendriform and tridendriform algebras, planar rooted trees, binarisation map, \AZVs{}, \MZVs{}.
	
	\tableofcontents
	
	\section*{Introduction}
	
	\addcontentsline{toc}{section}{Introduction}
	
	\subsection*{State of the art}
	
	\addcontentsline{toc}{subsection}{State of the art}
	
	\MZVs{} can be traced back to Euler~\cite{euler1776meditationes} and have since then been rediscovered under many names. However, the inception of their modern study should be attributed to Ecalle~\cite{Ecalle} (see~\cite{schneps2015ari} for an introduction in a more modern language). This sparked a systematic study of their various properties, with notable stepstones being laid by Hoffman~\cite{Ho92} and Zagier~\cite{Za94} among others. \MZVs{} and their generalisation have found applications to various areas of mathematical physics, in particular perturbative and integrable quantum field theory and string theory, see for example~\cite{todorov2014polylogarithms,Wi91,zagier2019genus}.
	
	The study of \MZVs{} remains today an active field of research and it is not the purpose of this article to make its survey. One should nonetheless mention the \emph{motivic} approach to \MZVs{} which uses methods of
	algebraic geometry~\cite{Andre_2009}. Some of the most important progress toward a full understanding of the structure of \MZVs{} have been made using the theory of motives, see for example~\cite{brown2012mixed}. In this theory, \MZVs{} are interpreted as particular \emph{motives}~\cite{Kont_2001}. Those motives can be evaluated thanks to the period morphism. The main challenge of the motives setup
	is to prove the \emph{period conjecture} stating that the period morphism is injective. 
        
	\medskip
	
	\MZVs{} have been generalised in many ways, but let us mention for completeness conical and elliptic zeta values~\cite{En13,GPZ13}, Witten's \MZVs{}~\cite{Wi91}, and finite zeta values~\cite{kaneko2019introduction}. The generalisation that interests us here is called \emph{Arborified Zeta Values}. They seem to have been introduced by Ecalle~\cite{Ecalle} and appeared in Yamamoto's work~\cite{Ya20} but their systematic study was started by Manchon~\cite{Manchon_16}. Notice that finite versions of \AZVs{} were introduced~\cite{On16}, and renormalised \AZVs{} have been investigated~\cite{CGPZ3}. 
	
	In Manchon's work~\cite{Manchon_16}, the author suggests two research directions:
	\begin{itemize}
	 \item to find a binarisation map (see next subsection) relating the simple and contracting arborifications which we call flattening maps in this paper (see definition~\ref{defi:flat_map} below);
	 \item to generalise properties of \MZVs{} to \AZVs{}.
	\end{itemize}
	The first question was recently solved by the third author~\cite{Fan25}. His construction will play a crucial role in this work and we give some details later on.
	
	Attempts to answer the second question have been made~\cite{clavier2020double,clavier2024generalisations} but not all relevant properties of \MZVs{} could be lifted to \AZVs{}. A new generalisation of \MZVs{} to planar rooted trees (Schroeder trees to be precise) was built using universal properties of these trees in the categories of dendriform and tridendriform algebras~\cite{catoire2025tridendriform}. These generalisations were then related to \AZVs{}, hence showing a deep relation between them and tridendriform and dendriform structures of planar rooted trees. Here we take the opposite point of view and build new tridendriform and dendriform products on planar rooted trees such that all relevant properties of \MZVs{} generalise to \AZVs{}; thus fully answering the second question of Manchon mentioned above~\cite{Manchon_16}.

	\subsection*{Elements of the theory of \MZVs{}}
	
	\addcontentsline{toc}{subsection}{Elements of the theory of \MZVs{}}	


\MZVs{} are real numbers defined, for $p\geq 1$ integers $n_1,\ldots,n_p$ with $n_i\in\N^*:=\Z_{\geq1}$ and $n_1\geq 2$, by the series
\[
\zeta(n_1,\ldots,n_p) \coloneqq \sum_{k_1>\cdots>k_p>0} \frac{1}{k_1^{n_1}\cdots k_p^{n_p}}.
\]
They admit a representation in terms of iterated integrals. For $k\in \N^*, (\epsilon_1,\cdots,\epsilon_k)\in \{x,y\}^k$ with $\epsilon_1=x$ and $\epsilon_k=y$ we set
	\[
	\zeta_\shuffle(\epsilon_1,\cdots,\epsilon_k)=\int_{1\geq t_1\geq \cdots \geq t_k\geq0}\omega_{\epsilon_1}(t_1)\cdots\omega_{\epsilon_k}(t_k)	
	\]
with $\omega_x(t)=dt/t$ and $\omega_y(t)=dt/(1-t)$. In the following, we will see these objects as maps. Let $\Wcv$ be the vector space generated by words written in the alphabet $\N^*$ whose first letter is not 1, and $\Wxycv$ the vector space generated by words written in the alphabet $\{x,y\}$ whose first letter is $x$ and last letter is $y$. Then we define $\zeta:\Wcv \longrightarrow\R$ and $\zeta_\shuffle:\Wxycv\longrightarrow\R$ with the formulas above, extended by linearity. Since by definition $\Wcv$ and $\Wxycv$ both contain the empty word $\emptyset$, we also set $\zeta(\emptyset)=\zeta_\shuffle(\emptyset)=1$.

The two representations of \MZVs{} are related through \emph{the binarisation map} mentioned above. This map $\fraks:\calW_{\N^*}\longrightarrow\calW_{\{x,y\}}$ is defined by $\fraks(\emptyset)=\emptyset$ and
		\begin{equation} \label{eq:bin_map_words}
		 \fraks(n_1\cdots n_p) \coloneqq (\underbrace{x\cdots x}_{n_1-1\text{ times}}y)\sqcup\ldots\sqcup(\underbrace{x\cdots x}_{n_p-1\text{ times}}y)
		\end{equation}
        extended by linearity, where $\sqcup$ is the usual concatenation product of words.
        Then 
        \begin{equation} \label{eq:Kontsevich}
         \fraks(\Wcv)\subseteq\Wxycv\quad\text{and}\quad\forall w\in\Wcv,~\zeta(w)=\zeta_\shuffle\circ\fraks(w).
        \end{equation}
        A proof of this relation between iterated series and iterated integrals representations of \MZVs{} was given in~\cite{Za94} quoting an unpublished observation by Kontsevich. This is why this relation is sometimes called Kontsevich's relation.
        
        \medskip

        The $\Q$-vector space spanned by \MZVs{} has
two complementary algebraic descriptions. On the one hand, the product of two series is controlled by the quasi-shuffle product. On the other hand, the iterated integral representation of \MZVs{} gives rise to the shuffle product. Let us now briefly recall the constructions of these products. We refer the readers to, for example,~\cite{Ho00,Waldschmidt} for details relevant to \MZVs{}.

Let $\Omega$ (resp. $(\Omega, +)$) be a set (resp. a commutative semigroup).  Let $\WO$ be the vector space spanned by words written in the alphabet $\Omega$. We introduce two products over $\WO$. The \emph{shuffle product} $\shuffle$ (resp. the 
  \emph{quasi-shuffle product} $\cshuffle$) is recursively defined by $\emptyset\shuffle w = w\shuffle\emptyset = w$ (resp. $\emptyset\cshuffle w = w\cshuffle\emptyset = w$) and for any $\omega,\omega' \in\Omega$, and any words $w$ and $w'$ written in the alphabet $\Omega$ 
  \begin{equation} \label{eq:shuffle} 
   \left((\omega)\sqcup w\right) \shuffle \left((\omega')\sqcup w'\right) = (\omega)\sqcup\left[w \shuffle \left((\omega')\sqcup w'\right) \right] + (\omega')\sqcup\left[\left((\omega)\sqcup w\right)\shuffle w'\right] 
  \end{equation}
  (resp. %
  \begin{equation} \label{eq:quasi-shuffle}
   \left((\omega) \sqcup w\right) \cshuffle \left((\omega') \sqcup w'\right) = (\omega) \sqcup \left[w \cshuffle \left((\omega') \sqcup w'\right) \right] + (\omega') \sqcup \left[\left((\omega) \sqcup w\right)\cshuffle w'\right]  + (\omega+\omega') \sqcup \left[w\cshuffle w'\right]\,),
  \end{equation}
  where $(\alpha)$ is a word containing only the letter $\alpha$.
  
  Now, take the semigroup $(\N^*,+)$ for iterated series and the set $\{x,y\}$ for iterated integrals. Then $\Wcv$ and $\Wxycv$ are subalgebras of $\W$ and $\Wxy$ for the products $\cshuffle$ and $\shuffle$ respectively. Furthermore, $\zeta$ and $\zeta_\shuffle$ are algebra morphisms for the products $\cshuffle$ and $\shuffle$ respectively:
if $u,v\in\Wcv$ and $u',v'\in\Wxycv$, then
	\begin{equation} \label{eq:shuffle_stuffle_zeta}
			\zeta(u\cshuffle v)=\zeta(u)\cdot \zeta(v),\qquad \zeta_{\shuffle}(u'\shuffle v')=\zeta_{\shuffle}(u')\cdot \zeta_{\shuffle}(v'). 
		\end{equation}
		These relations also stand for generalisations of \MZVs{}, see for example~\cite{racinet2002doubles}.
		 The consequences of relations~\eqref{eq:shuffle_stuffle_zeta} together with the iterated integral representations~\eqref{eq:Kontsevich} are called \emph{finite double shuffle relations}.
        
        Finally, for any $w\in\Wcv$, ${w\cshuffle (1)-\fraks^{-1}\left(\fraks(w)\shuffle(y)\right)\in\Wcv}$. Given $w\in\Wcv$, Hoffman's regularisation relation~\cite{Ho92,hoffman1997algebra} is 
        \begin{equation} \label{eq:Hoffman}
        	\quad w\cshuffle (1)-\fraks^{-1}\left(\fraks(w)\shuffle(y)\right) \in \Ker(\zeta).
        \end{equation}
        An important conjecture of the theory of \MZVs{} is
        \begin{conj} \label{conj:mzvs}
            Any rational relation among \MZVs{} can be derived from the finite double shuffle relations and Hoffman's regularisation relation.
        \end{conj}
        Here we use ``rational relations'' for ``polynomial relations with rational coefficients''. 
        For the relation between this conjecture and other conjectures about \MZVs{}, see~\cite{ihara2006derivation}.
        

        \subsection*{Content and main results}
        
        \addcontentsline{toc}{subsection}{Content and main result}

The paper is organised as follows. 
In section~\ref{sec:tridend}, we recall tridendriform and dendriform algebras (definitions~\ref{defi:tridend} and~\ref{defi:dend}). The relevant products on planar trees decorated by a set and a semigroup are introduced in definitions~\ref{defi:shuffle_trees} and~\ref{defi:quasishuffle_trees} respectively. The main results of the first section are propositions~\ref{prop:dend_struct} and~\ref{prop:tridend_struct} where we state that these products endow spaces of planar rooted trees with structures of dendriform and tridendriform algebras.

In section~\ref{sec:combs}, we recall the comb representations of planar rooted trees in definition~\ref{defi:combs}. We further recall the shuffle and quasi-shuffle maps (definition~\ref{def:quasi_shuffles}) and how they act on planar rooted trees (definition~\ref{def:quasiaction}). The main result of the second section is theorem~\ref{thm:comb} where a combinatorial (i.e. non inductive) description of the tridendriform and dendriform products is given. This construction is based on previous work by the first author \cite{Catoire_23} and also provides us with a description of the various products of the dendriform and tridendriform structures (corollary~\ref{cor:products}).

In section~\ref{sec:azvs} we start by introducing the planar \AZVs{} (definition~\ref{def:azvs}) and recall some results of the second author~\cite{clavier2020double} regarding these objects (theorem~\ref{thm:AZV_flaten}). The main result of this section, and one of the main results of this paper, is theorem~\ref{thm:azv_shuffle_stuffle}, where we show \AZVs{} are algebra morphisms for our new shuffle and quasi-shuffle products of rooted trees. 
It follows from the fact that flattening maps (definition~\ref{defi:flat_map}) are morphisms.


Section~\ref{sec:Hoffman} starts with a reminder of the binarisation map obtained by the third author~\cite{Fan25}. We first recall the minimal incomparable pair of a planar rooted tree (definition~\ref{def:min_incomp_pairs}) and its process tree (definition~\ref{defi:process_tree}). This allows us to define the relevant binarisation map in definition~\ref{defi:error_tree} which uses a previously built map (definition~\ref{defi:old_bin_map}). We then recall that this binarisation map solves the first question of Manchon~\cite{Manchon_16} (theorem~\ref{thm F.}) and notice in corollary~\ref{coro:arbo_Konts} that it implies that it relates our two versions of \AZVs{} (series and integrals). Then, in proposition~\ref{Prop:phi_bij} we prove that the $\phi$ map of the third author's work~\cite{Fan25} (definition~\ref{defi:error_tree}) is bijective. It is one of the main ingredients of the binarisation map.
The  crucial and main result of section~\ref{sec:Hoffman} is theorem~\ref{thm:arbo_Hoffman} about arborified versions of Hoffman's regularisation relation~\eqref{eq:Hoffman}. 
Note this result together with theorem~\ref{thm:azv_shuffle_stuffle} 
 and corollary~\ref{coro:arbo_Konts} which shows that the binarisation map relates the two versions of \AZVs{}, fully answer the second question of~\cite{Manchon_16} mentioned above.

Finally in section~\ref{sec:relating_zetas}, we state an arborified analogous of conjecture~\ref{conj:mzvs}~: conjecture~\ref{conj:azvs}. Notice that this conjecture requires some other obvious relations between \AZVs{}, the planarity relations, which are given in proposition~\ref{prop:planarity}. The third important result of this paper is theorem~\ref{thm:relating_conj}, where we show that our conjecture on \AZVs{} implies the standard conjecture on \MZVs{}. This, together with the fact that planar rooted trees are rather simple combinatorial objects with many powerful properties, suggests that studying the arborified case may be relevant to tackle conjecture~\ref{conj:mzvs}.

	\section{(Tri)dendriform structures for planar rooted trees} \label{sec:tridend}
	
	
	\subsection{Dendriform and tridendriform algebras}
	
	Let us start by defining the central structures of this paper, namely tridendriform and dendriform algebras. These definitions are taken from~\cite{burgunder2010tridendriform,foissy2007bidendriform,loday2002trialgebras,ronco2002eulerian}. 
	\begin{defi}\label{defi:tridend}
		Let $A$ be a vector space  endowed with three bilinear operations $\prec,\succ,\cdot$.  We say that $(A,\prec,\succ,\cdot)$ is a \emph{tridendriform algebra} if for all $(x,y,z)\in A^3$:
		\begin{align}
			(x\prec y)\prec z&=x\prec(y*z), \label{eq:tri1}\\
			(x\succ y)\prec z&=x\succ(y\prec z), \label{eq:tri2} \\
			(x* y)\succ z&=x\succ(y\succ z),  \label{eq:tri3} \\
			(x\succ y)\cdot z&=x\succ(y\cdot z),	 \label{eq:tri4} \\
			(x\prec y)\cdot z&=x\cdot(y\succ z), \label{eq:tri5} \\
			(x\cdot y)\prec z&=x\cdot(y\prec z), \label{eq:tri6} \\
			(x\cdot y)\cdot z&=x\cdot (y\cdot z), \label{eq:tri7}
		\end{align}
		where for all $x,y\in A$, we set $x*y\coloneqq x\prec y +x\succ y + x\cdot y$. We respectively call $\prec, \succ, \cdot$ the \emph{left, right and middle products}. 
	\end{defi}
	A dendriform algebra can be seen as a tridendriform algebra with a vanishing middle product.
	\begin{defi}\label{defi:dend} 
		We say that a tridendriform algebra $(A,\prec,\succ,\cdot)$ is a \emph{dendriform algebra} if $\cdot=0$. Hence, writing $\star\coloneqq\prec+\succ$, only three relations remain:
		\begin{align}
			(x\prec y)\prec z&=x\prec(y\star z), \label{eq:dend1}\\
			(x\succ y)\prec z&=x\succ(y\prec z), \label{eq:dend2} \\
			(x\star y)\succ z&=x\succ(y\succ z).  \label{eq:dend3} 
		\end{align}
	\end{defi}
	
	The notion of tridendriform comes with its notion of morphisms:
		\begin{defi}
			A \emph{tridendriform morphism} of tridendriform algebras between ${({A},\prec,\succ, \cdot)}$ and ${({B},\prec,\succ,\cdot)}$ is a linear map $f:A\rightarrow B$ such that for any $x,y\in A$:
			\begin{align*}
				f(x\prec y)=f(x)\prec f(y), && f(x\succ y)=f(x)\succ f(y), && f(x\cdot y)=f(x)\cdot f(y).
			\end{align*}
			If $\cdot=0$ we call such a map a \emph{dendriform morphism}.
		\end{defi}
	Note that (tri)dendriform algebras are endowed with a classical associative algebra structure. 
	\begin{Prop}
	 Let $(A,\prec,\succ,\cdot)$ (resp $(A,\prec,\succ)$) be a tridendriform algebra (resp. a dendriform algebra). Then the product $*$ on $A$ defined by $x*y\coloneqq x\prec y + x\succ y + x\cdot y$ (resp. $x\star y\coloneqq x\prec~y+x\succ~y$) is associative.
	\end{Prop}
	We recall classical and important examples of (tri)dendriform structures on words, given by the shuffle and quasi-shuffle of words recalled above.
	\begin{Eg}[shuffle product for words] \label{def:shuffle_words}
		We can split the shuffle product of words introduced before (see equation~\eqref{eq:shuffle}) into two smaller products $\prec$ and $\succ$ defined inductively by:
		\begin{align*}
			u\prec v = (u_1)\sqcup(u_2\dots u_n \shuffle v), &&
			u\succ v=(v_1)\sqcup(u\shuffle v_2\dots v_k).
		\end{align*}
	\end{Eg}
	For the tridendriform structure, the set $\Omega$ needs to have a semigroup structure. We denote by $+$ its product as we are mainly interested in $\Omega=(\N^*,+)$.
	\begin{Eg}[quasi-shuffle product for words]  \label{def:quasi_shuffle_words} 
		We split the quasi-shuffle product of words  defined in equation~\eqref{eq:quasi-shuffle} into three smaller products $\prec,\succ$ and $\cdot$ defined inductively by:
		\begin{align*}
			u\prec v = &(u_1)\sqcup(u_2\dots u_n \cshuffle v),\quad
			u\succ v=(v_1)\sqcup(u\cshuffle v_2\dots v_k), \\
			 &u \cdot v=(u_1+v_1)\sqcup(u_2\dots u_n \cshuffle v_2\dots v_k).
		\end{align*}
	\end{Eg}
	Shuffle and quasi-shuffle, with the decompositions presented above, are classical examples of dendriform and tridendriform algebras~\cite{Ebrahimi_Fard_2018}.
	\begin{Prop}
	 Let $\Omega$ be a set (resp. $(\Omega,+)$ a semigroup). Then $(\calW_\Omega,\prec,\succ)$ (resp. $\left(\calW_{\Omega},\prec,\succ,\cdot\right)$) is a dendriform algebra (resp. a tridendriform algebra).
	\end{Prop}
	
	\subsection{Dendriform structure for planar rooted trees} \label{subsec:dend}
	
	Recall that a \emph{planar tree} is an oriented connected planar graph, such that the orientation of the edges defines a partial order on the vertices, with a unique minimal element (the \emph{root}) for this partial order, and such that there exists a unique path from the root to any vertex of the tree. Vertices that are maximal for the partial order are called \emph{leaves}. A \emph{planar forest} is the planar concatenation of planar trees. For any tree $T$, we write $V(T)$ the set of vertices of $T$, and we sometimes use the short-hand notation $|T|$ for $|V(T)|$.
	
	Given a set $\Omega$, a planar $\Omega$-decorated tree is a planar tree together with a map ${d_T:V(T)\longrightarrow\Omega}$. We write $\PTO$ the set of non-empty planar rooted trees decorated (on their vertices) by $\Omega$, and $\calPT_\Omega$ the vector space spanned by $\PTO$. 
	Finally, we will simply say that $T$ is a tree to mean that $T$ is a planar decorated tree, in particular omitting its decoration map.
    
	\begin{Eg}
		Consider $|\Omega|=1$ (hence we do not need to write decorations), we obtain all of the following trees $t$ with three internal vertices (black nodes) and at three leaves (white nodes): 
		
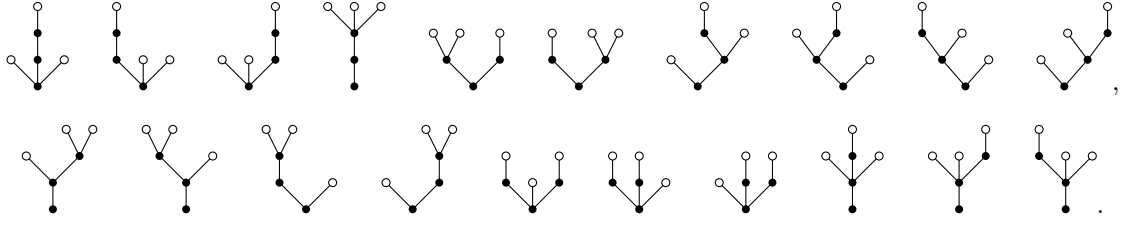
\begin{center}
		
			  \begin{tikzpicture}[x=1em,y=1em]
				\draw (0,0) -- (0,3);
				\draw (0,0) -- (-1,1);
				\draw (0,0) -- (1,1);
				\draw[color=black, fill=white](1,1) circle(0.15);
				\draw[color=black, fill=white](-1,1) circle(0.15);
				\draw[color=black, fill=white](0,3) circle(0.15);
				\fill(0,0) circle(0.15); 
				\fill (0,1) circle(0.15);
				\fill (0,2) circle(0.15);
			\end{tikzpicture}
            \hspace{1em}
			\begin{tikzpicture}[x=1em,y=1em]
				\draw (0,0) -- (0,1);
				\draw (-1,1) -- (-1,3);
				\draw (0,0) -- (-1,1);
				\draw (0,0) -- (1,1);
				\draw[color=black, fill=white](1,1) circle(0.15);
				\draw[color=black, fill=white](0,1) circle(0.15);
				\draw[color=black, fill=white](-1,3) circle(0.15);
				\fill(0,0) circle(0.15); 
				\fill (-1,1) circle(0.15);
				\fill (-1,2) circle(0.15);
			\end{tikzpicture}
			\hspace{1em}
			\begin{tikzpicture}[x=1em,y=1em]
				\draw (0,0) -- (0,1);
				\draw (1,1) -- (1,3);
				\draw (0,0) -- (-1,1);
				\draw (0,0) -- (1,1);
				\draw[color=black, fill=white](0,1) circle(0.15);
				\draw[color=black, fill=white](-1,1) circle(0.15);
				\draw[color=black, fill=white](1,3) circle(0.15);
				\fill(0,0) circle(0.15); 
				\fill (1,1) circle(0.15);
				\fill (1,2) circle(0.15);
			\end{tikzpicture}
			\hspace{1em}
			\begin{tikzpicture}[x=1em,y=1em]
				\draw (0,0) -- (0,3);
				\draw (0,2) -- (1,3);
				\draw (0,2) -- (-1,3);
				\draw[color=black, fill=white](0,3) circle(0.15);
				\draw[color=black, fill=white](1,3) circle(0.15);
				\draw[color=black, fill=white](-1,3) circle(0.15);
				\fill(0,0) circle(0.15); 
				\fill (0,1) circle(0.15);
				\fill (0,2) circle(0.15);
			\end{tikzpicture}
			\hspace{1em}
			\begin{tikzpicture}[x=1em,y=1em]
				\draw(0,2) -- (0.5,1);
				\draw(1,2) -- (0.5,1);
				\draw(2.5,2) -- (2.5,1);
				\draw(0.5,1) -- (1.5,0);
				\draw(2.5,1) -- (1.5,0);
				\draw[color=black, fill=white](0,2) circle(0.15);
				\draw[color=black, fill=white](1,2) circle(0.15);
				\draw[color=black, fill=white](2.5,2) circle(0.15);
				\fill(1.5,0) circle(0.15); 
				\fill (0.5,1) circle(0.15);
				\fill (2.5,1) circle(0.15);
			\end{tikzpicture}
			\hspace{1em}
			\begin{tikzpicture}[x=1em,y=1em]
				\draw(0.5,2) -- (0.5,1);
				\draw(2,2) -- (2.5,1);
				\draw(3,2) -- (2.5,1);
				\draw(0.5,1) -- (1.5,0);
				\draw(2.5,1) -- (1.5,0);
				\draw[color=black, fill=white](0.5,2) circle(0.15);
				\draw[color=black, fill=white](2,2) circle(0.15);
				\draw[color=black, fill=white](3,2) circle(0.15);
				\fill(1.5,0) circle(0.15); 
				\fill (0.5,1) circle(0.15);
				\fill (2.5,1) circle(0.15);
			\end{tikzpicture}
			\hspace{1em}
			\begin{tikzpicture}[x=1em,y=1em]
				\draw(-0.5,2) -- (-0.5,1);
				\draw(-0.5,1) -- (0.25,0);
				\draw(1,1) -- (0.25,0);
				\draw(0.25,0) -- (-0.75,-1);
				\draw(-1.75,0) -- (-0.75,-1);
				\draw[color=black, fill=white](-0.5,2) circle(0.15);
				\draw[color=black, fill=white](1,1) circle(0.15);
				\draw[color=black, fill=white](-1.75,0) circle(0.15);
				\fill(-0.75,-1) circle(0.15); 
				\fill (-0.5,1) circle(0.15);
				\fill (0.25,0) circle (0.15);
			\end{tikzpicture}
			\hspace{1em}
			\begin{tikzpicture}[x=1em,y=1em]
				\draw(0.5,2) -- (0.5,1);
				\draw(-1,1) -- (-0.25,0);
				\draw(0.5,1) -- (-0.25,0);
				\draw(-0.25,0) -- (0.75,-1);
				\draw(1.75,0) -- (0.75,-1);
				\draw[color=black, fill=white](0.5,2) circle(0.15);
				\draw[color=black, fill=white](-1,1) circle(0.15);
				\draw[color=black, fill=white](1.75,0) circle(0.15);
				\fill(0.75,-1) circle(0.15); 
				\fill (-0.25,0) circle(0.15);
				\fill (0.5,1) circle (0.15);
			\end{tikzpicture}
			\hspace{1em}
			\begin{tikzpicture}[x=1em,y=1em]
			\draw(-0.5,2) -- (-0.5,1);
			\draw(1,1) -- (0.25,0);
			\draw(-0.5,1) -- (0.25,0);
			\draw(0.25,0) -- (1.125, -1);
			\draw(2,0) -- (1.125, -1);
			\draw[color=black, fill=white](-0.5,2) circle(0.15);
			\draw[color=black, fill=white](1,1) circle(0.15);
			\draw[color=black, fill=white](2,0) circle(0.15);
			\fill(1.125,-1) circle(0.15); 
			\fill (0.25,0) circle(0.15);
			\fill (-0.5,1) circle(0.15);
			\end{tikzpicture}
			\hspace{1em}
			\begin{tikzpicture}[x=1em,y=1em]
				\draw(0.5,2) -- (0.5,1);
				\draw(-1,1) -- (-0.25,0);
				\draw(0.5,1) -- (-0.25,0);
				\draw(-0.25,0) -- (-1.125, -1);
				\draw(-2,0) -- (-1.125, -1);
				\draw[color=black, fill=white](0.5,2) circle(0.15);
				\draw[color=black, fill=white](-1,1) circle(0.15);
				\draw[color=black, fill=white](-2,0) circle(0.15);
				\fill(-1.125,-1) circle(0.15); 
				\fill (-0.25,0) circle(0.15);
				\fill (0.5,1) circle(0.15);
			\end{tikzpicture},
		\end{center}
        \begin{center}
			\begin{tikzpicture}[x=1em,y=1em]
				\draw (0,0) -- (0,1);
				\draw (0,1) -- (1,2);
				\draw (0,1) -- (-1,2);
				\draw (1,2) -- (0.5,3);
				\draw (1,2) -- (1.5,3);
				\draw[color=black, fill=white](0.5,3) circle(0.15);
				\draw[color=black, fill=white](1.5,3) circle(0.15);
				\draw[color=black, fill=white](-1,2) circle(0.15);
				\fill(0,0) circle(0.15);
				\fill (0,1) circle(0.15);
				\fill (1,2) circle(0.15);
			\end{tikzpicture}
			\hspace{1em}
			\begin{tikzpicture}[x=1em,y=1em]
				\draw (0,0) -- (0,1);
				\draw (0,1) -- (-1,2);
				\draw (0,1) -- (1,2);
				\draw (-1,2) -- (-0.5,3);
				\draw (-1,2) -- (-1.5,3);
				\draw[color=black, fill=white](-0.5,3) circle(0.15);
				\draw[color=black, fill=white](-1.5,3) circle(0.15);
				\draw[color=black, fill=white](1,2) circle(0.15);
				\fill(0,0) circle(0.15);
				\fill (0,1) circle(0.15);
				\fill (-1,2) circle(0.15);
			\end{tikzpicture}
			\hspace{1em}
			\begin{tikzpicture}[x=1em,y=1em]
				\draw (0,0) -- (1,1);
				\draw (0,0) -- (-1,1);
				\draw (-1,1) -- (-1,2);
				\draw (-1,2) -- (-0.5,3);
				\draw (-1,2) -- (-1.5,3);
				\draw[color=black, fill=white](-0.5,3) circle(0.15);
				\draw[color=black, fill=white](-1.5,3) circle(0.15);
				\draw[color=black, fill=white](1,1) circle(0.15);
				\fill(0,0) circle(0.15);
				\fill (-1,2) circle(0.15);
				\fill (-1,1) circle(0.15);
			\end{tikzpicture}
			\hspace{1em}
			\begin{tikzpicture}[x=1em,y=1em]
				\draw (0,0) -- (-1,1);
				\draw (0,0) -- (1,1);
				\draw (1,1) -- (1,2);
				\draw (1,2) -- (0.5,3);
				\draw (1,2) -- (1.5,3);
				\draw[color=black, fill=white](0.5,3) circle(0.15);
				\draw[color=black, fill=white](1.5,3) circle(0.15);
				\draw[color=black, fill=white](-1,1) circle(0.15);
				\fill(0,0) circle(0.15);
				\fill (1,2) circle(0.15);
				\fill (1,1) circle(0.15);
			\end{tikzpicture}
			\hspace{1em}
			\begin{tikzpicture}[x=1em,y=1em]
				\draw (0,0) -- (-1,1);
				\draw (0,0) -- (1,1);
				\draw (0,0) -- (0,1);
				\draw (1,1) -- (1,2);
				\draw (-1,1) -- (-1,2);
				\draw[color=black, fill=white](0,1) circle(0.15);
				\draw[color=black, fill=white](1,2) circle(0.15);
				\draw[color=black, fill=white](-1,2) circle(0.15);
				\fill(0,0) circle(0.15);
				\fill (-1,1) circle(0.15);
				\fill (1,1) circle(0.15);
			\end{tikzpicture}
			\hspace{1em}
			\begin{tikzpicture}[x=1em,y=1em]
				\draw (0,0) -- (-1,1);
				\draw (0,0) -- (1,1);
				\draw (0,0) -- (0,1);
				\draw (0,1) -- (0,2);
				\draw (-1,1) -- (-1,2);
				\draw[color=black, fill=white](1,1) circle(0.15);
				\draw[color=black, fill=white](0,2) circle(0.15);
				\draw[color=black, fill=white](-1,2) circle(0.15);
				\fill(0,0) circle(0.15);
				\fill (-1,1) circle(0.15);
				\fill (0,1) circle(0.15);
			\end{tikzpicture}
			\hspace{1em}
			\begin{tikzpicture}[x=1em,y=1em]
				\draw (0,0) -- (1,1);
				\draw (0,0) -- (-1,1);
				\draw (0,0) -- (0,1);
				\draw (0,1) -- (0,2);
				\draw (1,1) -- (1,2);
				\draw[color=black, fill=white](-1,1) circle(0.15);
				\draw[color=black, fill=white](0,2) circle(0.15);
				\draw[color=black, fill=white](1,2) circle(0.15);
				\fill(0,0) circle(0.15);
				\fill (1,1) circle(0.15);
				\fill (0,1) circle(0.15);
			\end{tikzpicture}
			\hspace{1em}
			\begin{tikzpicture}[x=1em,y=1em]
				\draw (0,0) -- (0,2);
				\draw (0,2) -- (0,3);
				\draw (0,1) -- (1,2);
				\draw (0,1) -- (-1,2);
				\draw[color=black, fill=white](0,3) circle(0.15);
				\draw[color=black, fill=white](1,2) circle(0.15);
				\draw[color=black, fill=white](-1,2) circle(0.15);
				\fill(0,0) circle(0.15); 
				\fill (0,1) circle(0.15);
				\fill (0,2) circle(0.15);
			\end{tikzpicture}
			\hspace{1em}
			\begin{tikzpicture}[x=1em,y=1em]
				\draw (0,0) -- (0,2);
				\draw (1,2) -- (1,3);
				\draw (0,1) -- (1,2);
				\draw (0,1) -- (-1,2);
				\draw[color=black, fill=white](0,2) circle(0.15);
				\draw[color=black, fill=white](1,3) circle(0.15);
				\draw[color=black, fill=white](-1,2) circle(0.15);
				\fill(0,0) circle(0.15); 
				\fill (0,1) circle(0.15);
				\fill (1,2) circle(0.15);
			\end{tikzpicture}
			\hspace{1em}
			\begin{tikzpicture}[x=1em,y=1em]
				\draw (0,0) -- (0,2);
				\draw (-1,2) -- (-1,3);
				\draw (0,1) -- (1,2);
				\draw (0,1) -- (-1,2);
				\draw[color=black, fill=white](-1,3) circle(0.15);
				\draw[color=black, fill=white](1,2) circle(0.15);
				\draw[color=black, fill=white](0,2) circle(0.15);
				\fill(0,0) circle(0.15); 
				\fill (0,1) circle(0.15);
				\fill (-1,2) circle(0.15);
			\end{tikzpicture}.
		\end{center}
	\end{Eg}
    Furthermore, given $T_1,\cdots,T_k, k$ elements of $\PT_\Omega$ and $x\in\Omega$, we write $B_+^x(T_1\cdots T_k)$ the tree obtained by grafting $T_1,\cdots,T_k$ in this order on a new vertex which is decorated by $x$. We extend $B_+^x$ by linearity to be defined on elements of $\calPTO$.
	
	
	\begin{defi} \label{defi:shuffle_trees}
	 Let $\Omega$ be a set. Let us define for any pair of trees $(S,T)$ in $\PT_\Omega\times \PT_\Omega$ the products $S<T$ and $S>T$ (and also $S\shuffle^T T:=S<T+S>T$) inductively on $|S|+|T|$. 
    \begin{enumerate}
	   \item $S<T$. If $S$ has only one vertex, we can assume  $S = \tdun{x}$. Then we set
	   \begin{equation} \label{eq:quasi_ini_dend}
	    \tdun{x}<T \coloneqq B_+^x(T).
	   \end{equation}
       If $S$ does not have only one vertex, we can assume $S=B_+^x(S_1\cdots S_k)$. Then we set
       $$S<T \coloneqq B_+^x(S_1\cdots(S_k\shuffle^T T)).$$
	   \item $S>T$. If $T$ has only one vertex, we can assume $T = \tdun{y}$. Then we set
	   \begin{equation} \label{eq:quasi_ini_dend_2}
	    S>\tdun{y} \coloneqq B_+^y(S).
	   \end{equation}
       If $T$ does not have only one vertex, we can assume $T=B_+^y(T_1\cdots T_\ell)$. Then we set
       $$S>T := B_+^y((S\shuffle^T T_1)\cdots T_\ell).$$
	  \end{enumerate}
	 We then extend these products by bilinearity to products on $\calPT_\Omega$.
	\end{defi}
	
	\begin{Eg}
	For instance, we give some products of trees decorated with elements of $\{x,y\}$:
	\begin{align*}
		\labpoint{x} \shuffle^T \labY{y}{y}{x} &= \labBY{x}{y}{y}{x} + \labYl{y}{x}{y}{x} + \labYl{y}{y}{x}{x}, \\
		\labY{y}{x}{y} \shuffle^T \labpoint{x} &= \labBY{x}{y}{x}{y} + \labYr{y}{x}{x}{y} + \labYr{y}{x}{y}{x}, \\
		\labY{y}{x}{y} \shuffle^T \labY{y}{y}{x} &= \labY{y}{x}{y} < \labY{y}{y}{x} + \labY{y}{x}{y} > \labY{y}{y}{x} \\
		&= B_+^{y} \left(\labpoint{x} \left(\labpoint{y} \shuffle^T \labY{y}{y}{x}\right)\right) + B_+^{y} \left( \left(\labY{y}{x}{y} \shuffle^T \labpoint{y}\right)\labpoint{x}\right) \\
		&=
		2~\raisebox{-0.4\height}{\begin{tikzpicture}[line cap=round,line join=round,>=triangle 45,x=0.4cm,y=0.4cm]
				\draw [line width=.5pt] (0.,0)-- (0,3);
				\draw [line width=.5pt] (0.,0)-- (-1.,1.);
				\draw [line width=.5pt] (0.,1)-- (1.,2.);
				\draw[right] (0,0) node {\footnotesize{y}};
				\draw[right] (-1,1) node {\footnotesize{x}};
				\draw[right] (0,1) node {\footnotesize{y}};
				\draw[right] (0,2) node {\footnotesize{y}};
				\draw[right] (0,3) node {\footnotesize{y}};
				\draw[above] (1,2) node {\footnotesize{x}};
				\filldraw (0,0) circle (1pt);
				\filldraw (0,1) circle (1pt);
				\filldraw (-1,1) circle (1pt);
				\filldraw (0,3) circle (1pt);
				\filldraw (0,2) circle (1pt);
				\filldraw (1,2) circle (1pt);
		\end{tikzpicture}} +
	\raisebox{-0.4\height}{\begin{tikzpicture}[line cap=round,line join=round,>=triangle 45,x=0.4cm,y=0.4cm]
			\draw [line width=.5pt] (0.,0)-- (0,3);
			\draw [line width=.5pt] (0.,0)-- (-1.,1.);
			\draw [line width=.5pt] (0.,2)-- (1.,3.);
			\draw[right] (0,0) node {\footnotesize{y}};
			\draw[right] (-1,1) node {\footnotesize{x}};
			\draw[right] (0,1) node {\footnotesize{y}};
			\draw[right] (0,2) node {\footnotesize{y}};
			\draw[right] (0,3) node {\footnotesize{y}};
			\draw[above] (1,3) node {\footnotesize{x}};
			\filldraw (0,0) circle (1pt);
			\filldraw (0,1) circle (1pt);
			\filldraw (-1,1) circle (1pt);
			\filldraw (0,3) circle (1pt);
			\filldraw (0,2) circle (1pt);
			\filldraw (1,3) circle (1pt);
	\end{tikzpicture}} +
 \raisebox{-0.4\height}{\begin{tikzpicture}[line cap=round,line join=round,>=triangle 45,x=0.4cm,y=0.4cm]
 		\draw [line width=.5pt] (0.,0)-- (0,3);
 		\draw [line width=.5pt] (0.,2)-- (-1.,3.);
 		\draw [line width=.5pt] (0.,0)-- (1.,1.);
 		\draw[left] (0,0) node {\footnotesize{y}};
 		\draw[right] (-1,3) node {\footnotesize{x}};
 		\draw[right] (0,1) node {\footnotesize{y}};
 		\draw[right] (0,2) node {\footnotesize{y}};
 		\draw[right] (0,3) node {\footnotesize{y}};
 		\draw[above] (1,1) node {\footnotesize{x}};
 		\filldraw (0,0) circle (1pt);
 		\filldraw (0,1) circle (1pt);
 		\filldraw (-1,3) circle (1pt);
 		\filldraw (0,2) circle (1pt);
 		\filldraw (0,3) circle (1pt);
 		\filldraw (1,1) circle (1pt);
 \end{tikzpicture}} + 
2~\raisebox{-0.4\height}{\begin{tikzpicture}[line cap=round,line join=round,>=triangle 45,x=0.4cm,y=0.4cm]
		\draw [line width=.5pt] (0.,0)-- (0,3);
		\draw [line width=.5pt] (0.,0)-- (1.,1.);
		\draw [line width=.5pt] (0.,1)-- (-1.,2.);
		\draw[left] (0,0) node {\footnotesize{y}};
		\draw[right] (0,1) node {\footnotesize{y}};
		\draw[right] (-1,2) node {\footnotesize{x}};
		\draw[right] (0,2) node {\footnotesize{y}};
		\draw[right] (0,3) node {\footnotesize{y}};
		\draw[above] (1,1) node {\footnotesize{x}};
		\filldraw (0,0) circle (1pt);
		\filldraw (0,1) circle (1pt);
		\filldraw (-1,2) circle (1pt);
		\filldraw (0,3) circle (1pt);
		\filldraw (0,2) circle (1pt);
		\filldraw (1,1) circle (1pt);
\end{tikzpicture}}
	\end{align*}
	\end{Eg}

    
	\begin{Rq}\label{Rq:not_extension_prec_succ}
		
		Let $\emptyset$ be the empty tree and let us denote $\overline{\PTO}\coloneqq\PTO\bigcup\{\emptyset\}$ and $\overline{\calPTO}\coloneqq \K\overline{\PTO} = \K\cdot\oplus\calPTO$.
		It is impossible to define properly $<$ and $>$ onto $\overline{\PTO}$ such that $\emptyset$ is a unit for $\shuffle^T$ and extends the dendriform structure. Indeed, defining $\emptyset < \emptyset$ and $\emptyset > \emptyset$ would be inconsistent with ${\shuffle^T= < + >}$ and the dendriform axioms. Nevertheless, $\emptyset\shuffle^T \emptyset =\emptyset$ is well defined and we can still extend the products $<$ and $>$ by defining 
        \begin{equation} \label{eq:conv_left_right}
	 \forall T \in\PTO,~ T<\emptyset=\emptyset> T=T,\quad \emptyset<T=T>\emptyset=0.
	\end{equation}
     In particular, we are now allowed to write $\tdun{x} = B_+^x(\emptyset)$ and $\emptyset \shuffle^T T = T \shuffle^T\emptyset = T$. This will allow us in the rest of this paper to not specify the case where one (or more) tree has only one vertex, except for the initialisation case. Indeed, this would require many easy subcases to consider. Instead, we allow ourselves to write $S=B_+^x(S_1\cdots S_k)$ even in the case $S=\tdun{x}$. In this case, one has to understand that $k=0$ and that $S_i\shuffle^T T=T$ for any $T\in\calPT_\Omega$. 
	\end{Rq}

	
	Notice that the next result does hold only for non-empty trees.
	\begin{Prop} \label{prop:dend_struct}
	 Let $\Omega$ be a set. Then $(\calPT_\Omega,<,>)$ is a dendriform algebra.
	\end{Prop}
	\begin{proof}
	 Let us prove equations \eqref{eq:dend1}, \eqref{eq:dend2} and \eqref{eq:dend3} on three trees $S$, $T$ and $U$ inductively on $|S|+|T|+|U|$.
	 
	 \begin{description} 
	 	\item[Initialisation:] if $|S|+|T|+|U|=3$, then there exist $x$, $y$ and $z$ in $\Omega$ such that $S=\tdun{x}$, $T=\tdun{y}$ and $U=\tdun{z}$. Then for equation \eqref{eq:dend1} we have
	 	\begin{equation*}
	 		(\tdun{x}<\tdun{y})<\tdun{z} = \tddeux{x}{y}<\tdun{z} = B_+^x(\tdun{y}\shuffle^T\tdun{z}) = \tdun{x}<(\tdun{y}\shuffle^T\tdun{z})
	 	\end{equation*}
	 	where we used respectively twice equation~\eqref{eq:quasi_ini_dend} as well as the definition of $<$. Using exactly the same strategy for equation~\eqref{eq:dend2} we find on the one hand
	 	\begin{equation*}
	 		(\tdun{x}>\tdun{y})<\tdun{z} = \tddeux{y}{x}<\tdun{z} = B_+^y(\tdun{x}\shuffle^T\tdun{z}).
	 	\end{equation*}
	 	On the other hand, using equation~\ref{eq:quasi_ini_dend_2} we also find
	 	\begin{equation*}
	 		\tdun{x}>(\tdun{y}<\tdun{z}) = \tdun{x}>\tddeux{y}{z} =B_+^y(\tdun{x}\shuffle^T\tdun{z}).
	 	\end{equation*}
	 	So equation \eqref{eq:dend2} holds for trees with only one vertex. Similarly, we also have
	 	\begin{equation*}
	 		\tdun{x}>(\tdun{y}>\tdun{z}) = \tdun{x}>\tddeux{z}{y} = B_+^z(\tdun{x}\shuffle^T\tdun{y})=(\tdun{x}\shuffle^T\tdun{y})>\tdun{z}.
	 	\end{equation*}
	 	So equation \eqref{eq:dend3} also holds for trees with only one vertex.
	 	
	 	\item[Heredity: ] now assume that for $N\geq3$ the relations~\eqref{eq:dend1}, \eqref{eq:dend2} and \eqref{eq:dend3} hold for any triplets of trees $(S,T,U)$ such that $|S|+|T|+|U|\leq N$. Then in particular, for any such triplet of trees, we have $(S\shuffle^T T)\shuffle^T U=S\shuffle^T(T\shuffle^T U)$. 
	 	
	 	Now, let $S$, $T$ and $U$ be a triplet of trees such that $|S|+|T|+|U|=N+1$. Then there exist $(x,y,z)\in\Omega^3,(k,\ell,m)\in\N^3$ and $k+\ell+m$ trees (possibly empty) such that
	 	\begin{equation*}
	 		S=B_+^x(S_1\cdots S_k),\qquad T=B_+^y(T_1\cdots T_\ell),\qquad U=B_+^z(U_1\cdots U_m).
	 	\end{equation*}
	 	We then have
	 	\begin{align*}
	 		(S<T)<U & = B_+^x\big(S_1\cdots(S_k\shuffle^T T)\big)<U \quad\text{by definition of }< \\
	 		& = B_+^x\big(S_1\cdots((S_k\shuffle^T T)\shuffle^T U)\big) \quad\text{by definition of }< \\
	 		& = B_+^x\big(S_1\cdots (S_k\shuffle^T (T\shuffle^T U))\big) \quad\text{by the induction hypothesis} \\
	 		& = S<(T \shuffle^T U) \quad\text{by definition of }< .
	 	\end{align*}
	 	Notice that this computation stays valid if $k=0$. In this case $S_k\shuffle^T T$ should be read as $T$ and the result still holds. So equation \eqref{eq:dend1} holds at rank $N+1$.
	 	
	 	Now, using first the definition of $>$ then the definition of $<$ we find on the one hand
	 	\begin{equation*}
	 		(S>T)<U = B_+^y\big((S\shuffle^T T_1)\cdots T_\ell\big)<U = B_+^y\big((S\shuffle^T T_1)\cdots (T_\ell\shuffle^T U)\big).
	 	\end{equation*}
	 	On the other hand, using first the definition of $<$ then the definition of $>$ we have
	 	\begin{equation*}
	 		S>(T<U) = S>B_+^y\big(T_1\cdots(T_\ell\shuffle^T  U)\big) = B_+^y\big((S\shuffle^T T_1)\cdots (T_\ell\shuffle^T U)\big)
	 	\end{equation*}
	 	so equation \eqref{eq:dend2} also holds at rank $N+1$. Notice that once again if $\ell=0$, then both side reduce to $B_+^y(S\shuffle^T U)$, so that equation \eqref{eq:dend2} also holds in this case.
	 	Finally, we have
	 	\begin{align*}
	 		S>(T>U) & = S>B_+^z\big((T\shuffle^T U_1)\cdots U_m\big) \quad\text{by definition of }> \\
	 		& = B_+^z\big((S\shuffle^T (T\shuffle^T U_1))\cdots U_m\big) \quad\text{by definition of }> \\
	 		& = B_+^z\big(((S\shuffle^T T)\shuffle^T U_1)\cdots U_m\big) \quad\text{by the induction hypothesis} \\
	 		& = (S\shuffle^T T)>U \quad\text{by definition of }>
	 	\end{align*}
	 	and again the computation still holds if $m=0$. So equation \eqref{eq:dend3} also holds in this case. \qedhere
	 \end{description}
	\end{proof}
	

\subsection{Tridendriform structure for planar rooted trees}

    Given a semigroup $(\Omega,+)$, we now give a tridendriform structure on the vector space of planar rooted trees $\calPTO$. This construction will closely follow the one described in subsection~\ref{subsec:dend}.
	\begin{defi} \label{defi:quasishuffle_trees}
	 Let $(\Omega,+)$ be a semigroup. 
    We define $\cshuffle$ by induction  over $\overline{\PTO}=\{\emptyset\}\bigcup \PTO$ and $<$, $>$ and $\bullet$ over $\PTO$ by:
    \begin{enumerate}
    	\item for all $T\in\overline{\PTO}, \emptyset \cshuffle^T T= T \cshuffle^T \emptyset = T$;
    	\item when $S=B_+^x(S_1\cdots S_k)$ (where $S_1\cdots S_k$ can be empty), we define:
    	\[
    	S < T\coloneqq B_+^x\left(S_1\cdots(S_k\cshuffle^T T)\right).
    	\]
    	\item when $T=B_+^{y}(T_1\cdots T_l)$ (where $T_1\cdots T_l$ can be empty), we define:
    	\[
    	S > T := B_+^y\left((S\cshuffle^T T_1)\cdots T_l\right).
    	\]
    	\item when $S=B_+^x(S_1\cdots S_k)$ and $T=B_+^{y}(T_1\cdots T_l)$ (where $S_1\cdots S_k$ and $T_1\cdots T_l$
    	can be empty), we define:
    	\[
    	S\bullet T:= B_+^{x+y}\left(S_1\cdots(S_k\cshuffle^T T_1)\cdots T_\ell\right),
    	\]
    \end{enumerate}
    where $\cshuffle^T=< + > +\bullet.$
	 We then extend $<,>$ and $\bullet$ into products on $\calPT_\Omega$ and $\cshuffle^T$ to a product on $\overline{\calPTO}$ by bilinearity.
	Note that we use the same conventions as in the dendriform case (equation \eqref{eq:conv_left_right} together with $S\bullet\emptyset=\emptyset\bullet S=0$ for any $S$ in $\PTO$) to cover the induction case.
	\end{defi}
    
    \begin{Rq}
        A dendriform product on $\calPTO$ presented in subsection~\ref{subsec:dend} was introduced for \emph{binary} rooted trees~\cite{catoire2025tridendriform}. A tridendriform product was also defined in the same paper for \emph{Schroeder} rooted trees, or (in the case of $\calPTN$) with one constraint on the decoration. The undecorated version of this tridendriform product, is related to the combinatorics of the associahedron, was previously introduced~\cite{chapoton2000bigebres,loday2002trialgebras}. A different generalisation to decorated trees was proposed in~\cite{Catoire_23}.
    \end{Rq}
    
    \begingroup%
    \allowdisplaybreaks%
		\begin{Eg}
			For instance, we compute some examples of such products below:
			\begin{align*}
				\labpoint{1} \cshuffle^T \labpoint{2} &= \tddeux{1}{2} + \tddeux{2}{1} + \labpoint{3}, \\
				\labY{2}{1}{3} \cshuffle^T \labY{3}{1}{2} &= 
				\raisebox{-0.4\height}{\begin{tikzpicture}[line cap=round,line join=round,>=triangle 45,x=0.4cm,y=0.4cm]
						\draw [line width=.5pt] (0.,0)-- (0,3);
						\draw [line width=.5pt] (0.,0)-- (-1.,1.);
						\draw [line width=.5pt] (0.,1)-- (1.,2.);
						\draw[right] (0,0) node {\footnotesize{2}};
						\draw[right] (-1,1) node {\footnotesize{1}};
						\draw[right] (0,1) node {\footnotesize{3}};
						\draw[right] (0,2) node {\footnotesize{1}};
						\draw[right] (0,3) node {\footnotesize{3}};
						\draw[above] (1,2) node {\footnotesize{2}};
						\filldraw (0,0) circle (1pt);
						\filldraw (0,1) circle (1pt);
						\filldraw (-1,1) circle (1pt);
						\filldraw (0,3) circle (1pt);
						\filldraw (0,2) circle (1pt);
						\filldraw (1,2) circle (1pt);
				\end{tikzpicture}} +
			\raisebox{-0.4\height}{\begin{tikzpicture}[line cap=round,line join=round,>=triangle 45,x=0.4cm,y=0.4cm]
					\draw [line width=.5pt] (0.,0)-- (0,3);
					\draw [line width=.5pt] (0.,0)-- (-1.,1.);
					\draw [line width=.5pt] (0.,1)-- (1.,2.);
					\draw[right] (0,0) node {\footnotesize{2}};
					\draw[right] (-1,1) node {\footnotesize{1}};
					\draw[right] (0,1) node {\footnotesize{3}};
					\draw[right] (0,2) node {\footnotesize{3}};
					\draw[right] (0,3) node {\footnotesize{1}};
					\draw[above] (1,2) node {\footnotesize{2}};
					\filldraw (0,0) circle (1pt);
					\filldraw (0,1) circle (1pt);
					\filldraw (-1,1) circle (1pt);
					\filldraw (0,3) circle (1pt);
					\filldraw (0,2) circle (1pt);
					\filldraw (1,2) circle (1pt);
			\end{tikzpicture}} +
				\raisebox{-0.4\height}{\begin{tikzpicture}[line cap=round,line join=round,>=triangle 45,x=0.4cm,y=0.4cm]
						\draw [line width=.5pt] (0.,0)-- (0,3);
						\draw [line width=.5pt] (0.,0)-- (-1.,1.);
						\draw [line width=.5pt] (0.,2)-- (1.,3.);
						\draw[right] (0,0) node {\footnotesize{2}};
						\draw[right] (-1,1) node {\footnotesize{1}};
						\draw[right] (0,1) node {\footnotesize{3}};
						\draw[right] (0,2) node {\footnotesize{3}};
						\draw[right] (0,3) node {\footnotesize{1}};
						\draw[above] (1,3) node {\footnotesize{2}};
						\filldraw (0,0) circle (1pt);
						\filldraw (0,1) circle (1pt);
						\filldraw (-1,1) circle (1pt);
						\filldraw (0,3) circle (1pt);
						\filldraw (0,2) circle (1pt);
						\filldraw (1,3) circle (1pt);
				\end{tikzpicture}} +
				\raisebox{-0.4\height}{\begin{tikzpicture}[line cap=round,line join=round,>=triangle 45,x=0.4cm,y=0.4cm]
						\draw [line width=.5pt] (0.,0)-- (0,3);
						\draw [line width=.5pt] (0.,2)-- (-1.,3.);
						\draw [line width=.5pt] (0.,0)-- (1.,1.);
						\draw[left] (0,0) node {\footnotesize{3}};
						\draw[right] (0,1) node {\footnotesize{1}};
						\draw[right] (0,2) node {\footnotesize{2}};
						\draw[right] (-1,3) node {\footnotesize{1}};
						\draw[right] (0,3) node {\footnotesize{3}};
						\draw[above] (1,1) node {\footnotesize{2}};
						\filldraw (0,0) circle (1pt);
						\filldraw (0,1) circle (1pt);
						\filldraw (-1,3) circle (1pt);
						\filldraw (0,2) circle (1pt);
						\filldraw (0,3) circle (1pt);
						\filldraw (1,1) circle (1pt);
				\end{tikzpicture}} + 
				\raisebox{-0.4\height}{\begin{tikzpicture}[line cap=round,line join=round,>=triangle 45,x=0.4cm,y=0.4cm]
						\draw [line width=.5pt] (0.,0)-- (0,3);
						\draw [line width=.5pt] (0.,0)-- (1.,1.);
						\draw [line width=.5pt] (0.,1)-- (-1.,2.);
						\draw[left] (0,0) node {\footnotesize{3}};
						\draw[right] (0,1) node {\footnotesize{2}};
						\draw[right] (-1,2) node {\footnotesize{1}};
						\draw[right] (0,2) node {\footnotesize{3}};
						\draw[right] (0,3) node {\footnotesize{1}};
						\draw[above] (1,1) node {\footnotesize{2}};
						\filldraw (0,0) circle (1pt);
						\filldraw (0,1) circle (1pt);
						\filldraw (-1,2) circle (1pt);
						\filldraw (0,3) circle (1pt);
						\filldraw (0,2) circle (1pt);
						\filldraw (1,1) circle (1pt);
				\end{tikzpicture}}+
			\raisebox{-0.4\height}{\begin{tikzpicture}[line cap=round,line join=round,>=triangle 45,x=0.4cm,y=0.4cm]
					\draw [line width=.5pt] (0.,0)-- (0,3);
					\draw [line width=.5pt] (0.,0)-- (1.,1.);
					\draw [line width=.5pt] (0.,1)-- (-1.,2.);
					\draw[left] (0,0) node {\footnotesize{3}};
					\draw[right] (0,1) node {\footnotesize{2}};
					\draw[right] (-1,2) node {\footnotesize{1}};
					\draw[right] (0,2) node {\footnotesize{1}};
					\draw[right] (0,3) node {\footnotesize{3}};
					\draw[above] (1,1) node {\footnotesize{2}};
					\filldraw (0,0) circle (1pt);
					\filldraw (0,1) circle (1pt);
					\filldraw (-1,2) circle (1pt);
					\filldraw (0,3) circle (1pt);
					\filldraw (0,2) circle (1pt);
					\filldraw (1,1) circle (1pt);
			\end{tikzpicture}} \\
			+
			\raisebox{-0.4\height}{\begin{tikzpicture}[line cap=round,line join=round,>=triangle 45,x=0.4cm,y=0.4cm]
				\draw [line width=.5pt] (0.,0)-- (0,2);
				\draw [line width=.5pt] (0.,0)-- (1.,1.);
				\draw [line width=.5pt] (0.,1)-- (-1.,2.);
				\draw[left] (0,0) node {\footnotesize{3}};
				\draw[right] (0,1) node {\footnotesize{2}};
				\draw[right] (-1,2) node {\footnotesize{1}};
				\draw[right] (0,2) node {\footnotesize{4}};
				\draw[above] (1,1) node {\footnotesize{2}};
				\filldraw (0,0) circle (1pt);
				\filldraw (0,1) circle (1pt);
				\filldraw (-1,2) circle (1pt);
				\filldraw (0,2) circle (1pt);
				\filldraw (1,1) circle (1pt);
				\end{tikzpicture}}+&
			\raisebox{-0.4\height}{\begin{tikzpicture}[line cap=round,line join=round,>=triangle 45,x=0.4cm,y=0.4cm]
			\draw [line width=.5pt] (0.,0)-- (0,2);
			\draw [line width=.5pt] (0.,1)-- (1.,2.);
			\draw [line width=.5pt] (0.,0)-- (-1.,1.);
			\draw[left] (0,0) node {\footnotesize{2}};
			\draw[right] (0,1) node {\footnotesize{3}};
			\draw[right] (1,2) node {\footnotesize{2}};
			\draw[right] (0,2) node {\footnotesize{4}};
			\draw[above] (-1,1) node {\footnotesize{1}};
			\filldraw (0,0) circle (1pt);
			\filldraw (0,1) circle (1pt);
			\filldraw (1,2) circle (1pt);
			\filldraw (0,2) circle (1pt);
			\filldraw (-1,1) circle (1pt);
			\end{tikzpicture}}+
		\raisebox{-0.4\height}{\begin{tikzpicture}[line cap=round,line join=round,>=triangle 45,x=0.4cm,y=0.4cm]
		\draw [line width=.5pt] (0,0) -- (0,2);
		\draw [line width=.5pt] (0,0) -- (-1,1);
		\draw [line width=.5pt] (0,0) -- (1,1);
		\filldraw (0,0) circle (1pt);
		\filldraw (0,1) circle (1pt);
		\filldraw (-1,1) circle (1pt);
		\filldraw (1,1) circle (1pt);
        \filldraw (0,2) circle (1pt);
		\draw (0,0) node[below] {\footnotesize{$5$}};
		\draw (-1,1) node[below] {\footnotesize{$1$}};
		\draw (0,1) node[left] {\footnotesize{$3$}};
		\draw (0,2) node[left] {\footnotesize{$1$}};
		\draw (1,1) node[below] {\footnotesize{$2$}};
	\end{tikzpicture}}+
		\raisebox{-0.4\height}{\begin{tikzpicture}[line cap=round,line join=round,>=triangle 45,x=0.4cm,y=0.4cm]
		\draw [line width=.5pt] (0,0) -- (0,2);
		\draw [line width=.5pt] (0,0) -- (-1,1);
		\draw [line width=.5pt] (0,0) -- (1,1);
		\filldraw (0,0) circle (1pt);
		\filldraw (0,1) circle (1pt);
		\filldraw (0,2) circle (1pt);
		\filldraw (-1,1) circle (1pt);
		\filldraw (1,1) circle (1pt);
		\draw (0,0) node[below] {\footnotesize{$5$}};
		\draw (-1,1) node[below] {\footnotesize{$1$}};
		\draw (0,1) node[left] {\footnotesize{$1$}};
		\draw (0,2) node[left] {\footnotesize{$3$}};
		\draw (1,1) node[below] {\footnotesize{$2$}};
		\end{tikzpicture}}+
		\raisebox{-0.4\height}{\begin{tikzpicture}[line cap=round,line join=round,>=triangle 45,x=0.4cm,y=0.4cm]
		\draw [line width=.5pt] (0,0) -- (0,1);
		\draw [line width=.5pt] (0,0) -- (-1,1);
		\draw [line width=.5pt] (0,0) -- (1,1);
		\filldraw (0,0) circle (1pt);
		\filldraw (0,1) circle (1pt);
		\filldraw (-1,1) circle (1pt);
		\filldraw (1,1) circle (1pt);
		\draw (0,0) node[below] {\footnotesize{$5$}};
		\draw (-1,1) node[left] {\footnotesize{$1$}};
		\draw (0,1) node[above] {\footnotesize{$4$}};
		\draw (1,1) node[right] {\footnotesize{$2$}};
        \end{tikzpicture}}+
			\raisebox{-0.4\height}{\begin{tikzpicture}[line cap=round,line join=round,>=triangle 45,x=0.4cm,y=0.4cm]
			\draw [line width=.5pt] (0.,0)-- (0,2);
			\draw [line width=.5pt] (0.,1)-- (1.,2.);
			\draw [line width=.5pt] (0.,0)-- (-1.,1.);
			\draw[left] (0,0) node {\footnotesize{2}};
			\draw[right] (0,1) node {\footnotesize{6}};
			\draw[right] (1,2) node {\footnotesize{2}};
			\draw[right] (0,2) node {\footnotesize{1}};
			\draw[above] (-1,1) node {\footnotesize{1}};
			\filldraw (0,0) circle (1pt);
			\filldraw (0,1) circle (1pt);
			\filldraw (1,2) circle (1pt);
			\filldraw (0,2) circle (1pt);
			\filldraw (-1,1) circle (1pt);
\end{tikzpicture}} +
			\raisebox{-0.4\height}{\begin{tikzpicture}[line cap=round,line join=round,>=triangle 45,x=0.4cm,y=0.4cm]
				\draw [line width=.5pt] (0.,0)-- (0,2);
				\draw [line width=.5pt] (0.,0)-- (1.,1.);
				\draw [line width=.5pt] (0.,1)-- (-1.,2.);
				\draw[left] (0,0) node {\footnotesize{3}};
				\draw[right] (0,1) node {\footnotesize{3}};
				\draw[right] (-1,2) node {\footnotesize{1}};
				\draw[right] (0,2) node {\footnotesize{3}};
				\draw[above] (1,1) node {\footnotesize{2}};
				\filldraw (0,0) circle (1pt);
				\filldraw (0,1) circle (1pt);
				\filldraw (-1,2) circle (1pt);
				\filldraw (0,2) circle (1pt);
				\filldraw (1,1) circle (1pt);
				\end{tikzpicture}}.
			\end{align*}
	\end{Eg}
	\endgroup%

	We then have a result which is the counterpart to proposition \ref{prop:dend_struct}. Its proof is similar to the proof of proposition \ref{prop:dend_struct} and we omit here to not purposelessly make this section too long.
	\begin{Prop} \label{prop:tridend_struct}
	 Let $(\Omega,+)$ be a semigroup. Then $(\calPT_\Omega,<,>,\bullet)$ is a tridendriform algebra.
	\end{Prop}

	\section{Description with the comb representations} \label{sec:combs}
	
	\subsection{Comb representation of trees and quasi-shuffles}
	
	 Fix $\Omega$ a set of decorations and let $t$ be a tree of $\PTO$. It can always be seen as a comb:
	\begin{center}
		\raisebox{-0.5\height}
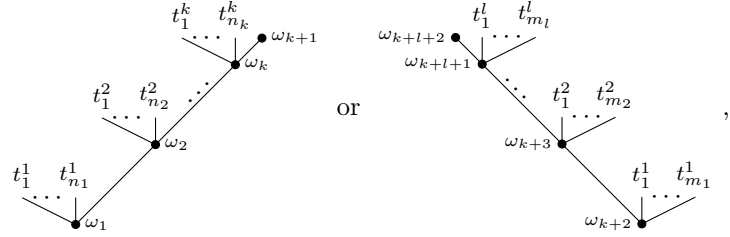
{\begin{tikzpicture}[line cap=round,line join=round,x=1em,y=1 em]
				\filldraw (0,1) circle (0.15);
				\draw (0,1)--(-2,2)
				node[above]{\footnotesize{$t^1_1$}};
				\draw (0,1)--(0,2) node[above]{\footnotesize{$t^1_{n_1}$}};
				\draw (-1,2) node{$\cdots$};
				\draw (0,1)--(7,8);
				\draw (3,4)--(1,5) node[above]{\footnotesize{$t^2_1$}};
				\filldraw (3,4) circle (0.15);
				\draw (3,4)--(3,5) node[above]{\footnotesize{$t^2_{n_2}$}};
				\draw (2,5) node{$\cdots$};
				\draw (4.7,6) node[rotate=45]{$\cdots$};
				\filldraw(7,8) circle(0.15);
				\filldraw (6,7) circle (0.15);
				\draw (6,7)--(4,8) node[above]{\footnotesize{$t^k_1$}};
				\draw (6,7)--(6,8) node[above]{\footnotesize{$t^k_{n_k}$}};
				\draw (5,8) node {$\cdots$};
				\draw (0,1) node[right] {\scriptsize{$\omega_1$}};
				\draw (3,4) node[right] {\scriptsize{$\omega_2$}};
				\draw (6,7) node[right] {\scriptsize{$\omega_{k}$}};
				\draw (7,8) node[right] {\scriptsize{$\omega_{k+1}$}};
		\end{tikzpicture}} or
		\raisebox{-0.5\height}{\begin{tikzpicture}[line cap=round,line join=round,x=1em,y=1 em]
				\filldraw (0,1) circle (0.15);
				\draw (0,1)--(2,2)
				node[above]{\footnotesize{$t^1_{m_1}$}};
				\draw (0,1)--(0,2) node[above]{\footnotesize{$t^1_{1}$}};
				\draw (1,2) node{$\cdots$};
				\draw (0,1)--(-7,8);
				\draw (-3,4)--(-1,5) node[above]{\footnotesize{$t^2_{m_2}$}};
				\filldraw (-3,4) circle (0.15);
				\draw (-3,4)--(-3,5) node[above]{\footnotesize{$t^2_{1}$}};
				\draw (-2,5) node{$\cdots$};
				\draw (-4.7,6) node[rotate=-45]{$\cdots$};
				\filldraw(-7,8) circle(0.15);
				\filldraw (-6,7) circle (0.15);
				\draw (-6,7)--(-4,8) node[above]{\footnotesize{$t^{l}_{m_{l}}$}};
				\draw (-6,7)--(-6,8) node[above]{\footnotesize{$t^{l}_{1}$}};
				\draw (-5,8) node {$\cdots$};
				\draw (0,1) node[left] {\scriptsize{$\omega_{k+2}$}};
				\draw (-3,4) node[left] {\scriptsize{$\omega_{k+3}$}};
				\draw (-6,7) node[left] {\scriptsize{$\omega_{k+l+1}$}};
				\draw (-7,8) node[left] {\scriptsize{$\omega_{k+l+2}$}};
		\end{tikzpicture}},
	\end{center}
	where $k+1$ (resp. $l+1$) is the number of nodes on the rightmost (resp. leftmost) branch of $t$, for $i\in\IEM{1}{k} \coloneqq \{1, 2, \ldots, k\}, n_i+1$ is the number of sons of the $i$-th node of this branch (resp. for $j\in\IEM{1}{l}, m_j+1$ is the number of sons of the $j$-th node of this branch) for all $j\in\IEM{1}{k+l+2}, \omega_j\in\Omega$.
	\begin{Not}
		Let $F=t_1\dots t_n$ be a planar forest composed of $n$ trees. For simplicity, we write:
		\[
		\raisebox{-0.3\height}{\begin{tikzpicture}[line cap=round,line join=round,x=0.3cm,y=0.3cm]
				\draw (0,1)--(-1,2) node[left]{$F$};
				\draw (0,1)--(1,2);
				\filldraw (1,2) circle (0.15);
				\fill (0,1) circle(0.15);
				\draw (0,1) node[left, below] {\footnotesize{$\omega$}};
				\draw (1,2) node[right] {\footnotesize{$\alpha$}};
		\end{tikzpicture}}  \text{ instead of }
		\raisebox{-0.5\height}{\begin{tikzpicture}[line cap=round,line join=round,x=0.3cm,y=0.3cm]
				\draw (0,1)--(-2,2) node[left,above]{$t_1$};
				\draw (0,1)--(0,2) node[right,above]{$t_{n}$};
				\draw (-1,2) node{$\cdots$};
				\draw (0,1)--(2,2);
				\filldraw (2,2) circle (0.15);
				\fill (0,1) circle(0.15);
				\draw (0,1) node[left,below] {\footnotesize{$\omega$}};
				\draw (2,2) node[right] {\footnotesize{$\alpha$}};
		\end{tikzpicture}}
		\text{ and }
		\raisebox{-0.3\height}{\begin{tikzpicture}[line cap=round,line join=round,x=0.3cm,y=0.3cm]
				\draw (0,1)--(1,2) node[right]{$F$};
				\draw (0,1)--(-1,2);
				\filldraw (-1,2) circle (0.15);
				\fill (0,1) circle(0.15);
				\draw (0,1) node[left, below] {\footnotesize{$\omega$}};
				\draw (-1,2) node[left] {\footnotesize{$\alpha$}};
		\end{tikzpicture}}  \text{ instead of }
		\raisebox{-0.5\height}{\begin{tikzpicture}[line cap=round,line join=round,x=0.3cm,y=0.3cm]
				\draw (0,1)--(2,2) node[left,above]{$t_{n}$};
				\draw (0,1)--(0,2) node[right,above]{$t_{1}$};
				\draw (1,2) node{$\cdots$};
				\draw (0,1)--(-2,2);
				\filldraw (-2,2) circle (0.15);
				\fill (0,1) circle(0.15);
				\draw (0,1) node[left, below] {\footnotesize{$\omega$}};
				\draw (-2,2) node[left] {\footnotesize{$\alpha$}};
		\end{tikzpicture}}.
		\]
	\end{Not}
	\begin{defi} \label{defi:combs}
	With these notations, we notice that all trees $t$ admit respectively a \emph{right comb representation} and a \emph{left comb representation}:
	\begin{align}\label{eq:tree_combs}
		t&=\decombright{F_1}{F_2}{F_k}{\omega_1}{\omega_2}{\omega_{k}}{\omega_{k+1}}
		& \text{ and } &&  t=\decombleft{F_{k+2}}{F_{k+3}}{F_{k+l+1}}{\omega_{k+2}}{\omega_{k+3}}{\omega_{k+l+1}}{\omega_{k+l+2}},
	\end{align}
	where $F_1,\dots,F_k$ and $F_{k+2},\dots, F_{k+l+1}$ are forests, possibly empty . In the following, we will consider that there exist two further forests $F_{k+1}$ and $F_{k+l+2}$, both empty and respectively grafted to the nodes decorated by $\omega_{k+1}$ and $\omega_{k+l+2}$.
		
	\end{defi}
	
	
	\begin{defi}[Quasi-shuffle]\label{def:quasi_shuffles}
		Let $k,l\in\N\setminus\{0\}$. A \emph{$(k,l)$-\qsh} is a surjective map $\sigma:\IEM{1}{k+l}\twoheadrightarrow\IEM{1}{n}$ for some positive integer $n$
		such that:
		\[
		\sigma(1)<\cdots<\sigma(k) \text{  and  } \sigma(k+1)<\cdots<\sigma(k+l).
		\]
		We will denote $\QSh(k,l)$ the set of all $(k,l)$-quasi-shuffles. When an element of $\QSh(k,l)$ is a bijection we call it a \emph{$(k,l)$-shuffle}. We denote the set of $(k,l)$-shuffles as $\Sh(k,l)$.
		
		Let $(k,l)\in(\N^*)^2,$ given $\sigma\in\QSh(k,l)$, we will denote it by $(\sigma(1),\ldots,\sigma(k+1),\ldots,\sigma(k+l))$.
	\end{defi}
	
	\begin{Eg}
		The map $\sigma:\IEM{1}{4} \rightarrow \IEM{1}{3}, 1 \mapsto 1, 2 \mapsto 3, 3 \mapsto 1,4 \mapsto 2$ is an element of $\QSh(2,2)$. It is denoted by $\sigma=(1,3,1,2).$
	\end{Eg}
	\begin{defi}\label{def:quasiaction} 
			Let $t$ and $s$ be two trees decorated by a semigroup $(\Omega,+)$ which respective right comb representation and left comb representation are given above (equation~\eqref{eq:tree_combs}).
			We denote by $k+1$ (resp. $l+1$) the number of nodes on the rightmost (resp. leftmost) branch of $t$ (resp. $s$). Let $\sigma$ be a $(k+1,l+1)$-quasi-shuffle  which has for image $\IEM{1}{n}$ with $n$ a non-negative integer. We build a tree from $\sigma$:
			\begin{enumerate}
				\item  We first consider the ladder with $n$ nodes:
				\begin{center}
					\raisebox{-0.5\height}
					
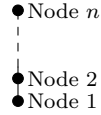
{\begin{tikzpicture}[line cap=round,line join=round,>=triangle 45,x=0.3cm,y=0.3cm]
							\begin{scope}
								\draw (0,1)--(0,2.5);
								\draw[dashed] (0,2.5)--(0,5);
								\filldraw [black] (0,1) circle (2pt) node[anchor=west]{\footnotesize{Node $1$}};
								\filldraw [black] (0,2) circle (2pt) node[anchor=west]{\footnotesize{Node $2$}};
								\filldraw [black] (0,5) circle (2pt) node[anchor=west]{\footnotesize{Node $n$}};
							\end{scope}
					\end{tikzpicture}}.
				\end{center}
				\item For $i\in\IEM{1}{k+1},$ we graft $F_i$ (with its decorations) as the \emph{left} son at the node $\sigma(i)$.
				\item For $i\in\IEM{k+2}{k+l+2},$ we graft $F_i$ with its decorations as the \emph{right} son at the node $\sigma(i)$.
				\item For $j\in\IEM{1}{n}$, the node $j$ drawn above is decorated by:
				\begin{itemize}
					\item $\omega_s$ if $\sigma^{-1}(\{j\})=\{s\}$ with $s\in\IEM{1}{k+1}$;
					\item $\omega_r$ if $\sigma^{-1}(\{j\})=\{r\}$ with $r\in\IEM{k+2}{k+l+2}$;
					\item $\omega_s + \omega_r$ if $\sigma^{-1}(\{j\})=\{s,r\}$ with $(s,r)\in\IEM{1}{k+1}\times \IEM{k+2}{k+l+2}$.
				\end{itemize}
			\end{enumerate}
			We denote the tree obtained from this construction $\sigma(t,s)$. 
			If $\sigma$ is a $(k+1,l+1)$-shuffle, we only require for $\Omega$ to be a set.
		\end{defi}
	
	\begin{Eg}
		Consider $t=\labY{2}{1}{3}$, $s=\labY{3}{1}{2}$ and a $(2,2)$-quasi-shuffle $\sigma=(1,3,1,2)$. Hence:
		\[
		t=\raisebox{-0.4\height}{\begin{tikzpicture}[line cap=round,line join=round,>=triangle 45,x=0.4cm,y=0.4cm]
			\draw [line width=.5pt] (0.,1.)-- (-1.,2.);
			\draw [line width=.5pt] (0.,1.)-- (1.,2.);
			\draw[above] (0,1) node {\footnotesize{$2$}};
			\draw[above] (1,2) node {\footnotesize{3}};
			\draw[above] (-1,2) node {$F_1$};
			\filldraw (0,1) circle (1pt);
			\filldraw (1,2) circle (1pt);
			\end{tikzpicture}} \text{ with } F_1=\labpoint{1} \text{ and } s=\raisebox{-0.4\height}{\begin{tikzpicture}[line cap=round,line join=round,>=triangle 45,x=0.4cm,y=0.4cm]
			\draw [line width=.5pt] (0.,1.)-- (-1.,2.);
			\draw [line width=.5pt] (0.,1.)-- (1.,2.);
			\draw[above] (0,1) node {\footnotesize{$3$}};
			\draw[above] (-1,2) node {\footnotesize{1}};
			\draw[above] (1,2) node {$F_2$};
			\filldraw (0,1) circle (1pt);
			\filldraw (-1,2) circle (1pt);
		\end{tikzpicture}} \text{ with } F_2=\labpoint{2}, \text{ thus } 
	\sigma(t,s) =\raisebox{-0.4\height}{\begin{tikzpicture}[line cap=round,line join=round,>=triangle 45,x=0.4cm,y=0.4cm]
	\draw [line width=.5pt] (0,0) -- (0,2);
	\draw [line width=.5pt] (0,0) -- (-1,1);
	\draw [line width=.5pt] (0,0) -- (1,1);
	\filldraw (0,0) circle (1pt);
	\filldraw (0,1) circle (1pt);
	\filldraw (0,2) circle (1pt);
	\filldraw (-1,1) circle (1pt);
	\filldraw (1,1) circle (1pt);
	\draw (0,0) node[below] {\footnotesize{$5$}};
	\draw (-1,1) node[below] {\footnotesize{$1$}};
	\draw (0,1) node[left] {\footnotesize{$1$}};
	\draw (0,2) node[left] {\footnotesize{$3$}};
	\draw (1,1) node[below] {\footnotesize{$2$}};
\end{tikzpicture}}.
		\]

	\end{Eg}
	
	\subsection{Combs and the products of planar trees}

	We can now prove the main result of this section, which we will need later on
	\begin{thm} \label{thm:comb}
		Let $\Omega$ be a set (resp. $(\Omega,+)$ be a semigroup).
	 The products $\shuffle^T$  (resp. $\cshuffle^T$ ) of $\calPTO$ are defined explicitly for any pair $(t,s)\in\calPTO^2$ by:
	 \begin{align*}
	 	t \shuffle^T s = \sum_{\sigma\in\Sh(k+1,l+1)} \sigma(t,s) \quad
	 	\left(\text{\normalfont{ resp. }} 
	 	 t \cshuffle^T s = \sum_{\sigma\in\QSh(k+1,l+1)} \sigma(t,s) \right)
	 \end{align*}
	 where $k+1$ (resp. $l+1$) is the number of nodes in the rightmost (resp. leftmost) branch of $t$ (resp. of $s$).
	\end{thm}
	\begin{proof}
		To prove this theorem, one just needs to check that these combinatorial descriptions satisfy the inductive constructions of definitions~\ref{defi:shuffle_trees} and~\ref{defi:quasishuffle_trees}. We only check the ones of definition~\ref{defi:quasishuffle_trees}. The other ones are shown similarly setting the product $\bullet$ to zero. To prove the combinatorial description of $t\cshuffle^T s$, we do an inductive proof over the sum of $k+1$ and $l+1$  given in equation~\ref{eq:tree_combs}.
		\begin{description}
			\item[Initialisation:] to initialise, we begin with $k+l=0,$ that is to say $k+1=1$ and $l+1=1$  as our trees are non-empty.
			As a consequence, $t=\labpoint{\omega}$ and $s=\labpoint{\alpha}$ 
			\begin{align*}
				(1,2)\left(\labpoint{\omega},\labpoint{\alpha}\right) + (2,1)\left(\labpoint{\omega},\labpoint{\alpha}\right) + (1,1)\left(\labpoint{\omega},\labpoint{\alpha}\right) &=
				\tddeux{\omega}{\alpha} + \tddeux{\alpha}{\omega} +\labpoint{\omega+\alpha}.
			\end{align*} 
			Comparing it with the definition of $\cshuffle^T$, one has :
			\begin{align*}
				\labpoint{\omega} \cshuffle^T \labpoint{\alpha}=B_{+}^{\omega}\left(\emptyset \cshuffle^T \labpoint{\alpha}\right)+ B_{+}^{\alpha}\left(\labpoint{\omega}\cshuffle^T \emptyset \right) + B_{+}^{\omega+\alpha}\left(\emptyset \cshuffle^T \emptyset\right) = \sum_{\sigma\in\QSh(1,1)} \sigma(t,s).
			\end{align*}
			\item[Heredity:] let $k,l\in\N^*$ such that $k+l\geq 1$. Let us suppose for all couple of trees $(t',s')$ with $k'+1$ the number of nodes on the path from the root to the rightmost leaf of $t'$ and $l'+1$ the number of nodes on the path from the root to the leftmost leaf of $s'$ such that $k'+l'<k+l$, the description of the theorem is valid. Then :%
			\begingroup%
			\allowdisplaybreaks%
			\begin{align*}
				&t\cshuffle^T s=t< s+t> s+t\bullet s
				= \decombright{F_1}{F_2}{F_k}{\omega_1}{\omega_2}{\omega_{k}}{\omega_{k+1}} < \decombleft{F_{k+2}}{F_{k+3}}{F_{k+l+1}}{\omega_{k+2}}{\omega_{k+3}}{\omega_{k+l+1}}{\omega_{k+l+2}} \\
				+& \decombright{F_1}{F_2}{F_k}{\omega_1}{\omega_2}{\omega_{k}}{\omega_{k+1}} > \decombleft{F_{k+2}}{F_{k+3}}{F_{k+l+1}}{\omega_{k+2}}{\omega_{k+3}}{\omega_{k+l+1}}{\omega_{k+l+2}} \\
				+& \decombright{F_1}{F_2}{F_k}{\omega_1}{\omega_2}{\omega_{k}}{\omega_{k+1}} \bullet \decombleft{F_{k+2}}{F_{k+3}}{F_{k+l+1}}{\omega_{k+2}}{\omega_{k+3}}{\omega_{k+l+1}}{\omega_{k+l+2}} \\ 
				=& B_+^{\omega_1}\left(F_1\left( \decombright{F_2}{F_3}{F_k}{\omega_{2}}{\omega_3}{\omega_{k}}{\omega_{k+1}} \cshuffle^T s\right)\right) \\
				+& B_+^{\omega_{k+2}}\left(\left(t \cshuffle^T\decombleft{F_{k+3}}{F_{k+4}}{F_{k+l+1}}{\omega_{k+3}}{\omega_{k+4}}{\omega_{k+l+1}}{\omega_{k+l+2}}\right) F_{k+2}\right) \\
				+&B_+^{\omega_1+ \omega_{k+2}}\left( F_{1} \left(\decombright{F_2}{F_3}{F_k}{\omega_2}{\omega_3}{\omega_{k}}{\omega_{k+1}} \cshuffle^T \decombleft{F_{k+3}}{F_{k+4}}{F_{k+l+1}}{\omega_{k+3}}{\omega_{k+4}}{\omega_{k+l+1}}{\omega_{k+l+2}}\right) F_{k+2}\right).
			\end{align*}
			\endgroup%
			Thanks to our induction hypothesis one gets:
			\begingroup%
			\allowdisplaybreaks%
			\begin{align*}
				&t \cshuffle^T s 
				= B_+^{\omega_1} \left(\sum_{\tau\in\QSh(k,l+1)} F_1\tau\left(\decombright{F_2}{F_3}{F_k}{\omega_2}{\omega_3}{\omega_k}{\omega_{k+1}},s \right) \right) \\
				+&B_+^{\omega_{k+2}}\left( \sum_{\delta\in\QSh(k+1,l)} \delta\left(t, \decombleft{F_{k+3}}{F_{k+4}}{F_{k+l+1}}{\omega_{k+3}}{\omega_{k+4}}{\omega_{k+l+1}}{\omega_{k+l+2}} \right) F_{k+2} \right)   \\
				+& B_+^{\omega_1+\omega_{k+2}}\left( \sum_{\gamma\in\QSh(k,l)} F_1 \gamma\left(\decombright{F_2}{F_3}{F_k}{\omega_2}{\omega_3}{\omega_k}{\omega_{k+1}},\decombleft{F_{k+3}}{F_{k+4}}{F_{k+l+1}}{\omega_{k+3}}{\omega_{k+4}}{\omega_{k+l+1}}{\omega_{k+l+2}}\right) F_{k+2} \right) \\
				=&\sum_{\substack{\tau\in \QSh(k,l+1) \\ \sigma_1=(1,\tau(1)+1,\cdots,\tau(k+l+1)+1)}}\sigma_1(t,s) 
				+\sum_{\substack{\delta\in \QSh(k+1,l) \\ \sigma_2=(\delta(1)+1,\cdots,\delta(k+1)+1,1,\delta(k+2)+1,\cdots,\delta(k+l+1)+1)}}\sigma_2(t,s) \\
				+&\sum_{\substack{\gamma\in \QSh(k,l) \\ \sigma_3=(1,\gamma(1)+1,\cdots,\gamma(k)+1,1,\gamma(k+1)+1,\cdots,\gamma(k+l)+1)}}\sigma_3(t,s) \\
				=& \sum_{\sigma\in \QSh(k+1,l+1)} \sigma(t,s). 
			\end{align*}%
			\endgroup%
		\end{description}
		This ends the proof of the theorem.		
	\end{proof}
	With the same idea, one can deduce 
	\begin{Cor} \label{cor:products}
		With notations of theorem~\ref{thm:comb}, dendriform products of definition~\ref{defi:shuffle_trees} are:
		\begin{equation*}
			t < s = \sum_{\substack{\sigma\in\Sh(k+1,l+1) \\ \sigma^{-1}(\{1\})=\{1\}}} \sigma(t,s)\quad \text{and} \quad
			t > s = \sum_{\substack{\sigma\in\Sh(k+1,l+1) \\ \sigma^{-1}(\{1\})=\{k+2\}}} \sigma(t,s).
		\end{equation*}
		and respectively, for the tridendriform products of definition~\ref{defi:quasishuffle_trees}
		\begin{align*}
			t < s &= \sum_{\substack{\sigma\in\QSh(k+1,l+1) \\ \sigma^{-1}(\{1\})=\{1\}}} \sigma(t,s),\qquad
			t > s = \sum_{\substack{\sigma\in\QSh(k+1,l+1) \\ \sigma^{-1}(\{1\})=\{k+2\}}} \sigma(t,s). \\
			&\text{ \normalfont{ and }} \quad t \bullet s = \sum_{\substack{\sigma\in\QSh(k+1,l+1) \\ \sigma^{-1}(\{1\})=\{1,k+2\}}} \sigma(t,s).
		\end{align*}
	\end{Cor}

\section{Arborified zeta values and (tri)dendriform products} \label{sec:azvs} 

\subsection{Arborified zeta values}

\AZVs{} were introduced in~\cite{Manchon_16,Ya20}. Like \MZVs{} they exist in a series and in an integral version.
\begin{defi}\label{def:azvs}
    Let $\calPTNconv$ be the subspace of $\calPTN$ spanned by trees whose roots are \emph{not} decorated by 1. We define the linear map $\zetaT:\calPTNconv \rightarrow  \R$ such that for any tree $t$ of $\calPTNconv$:
	\[
	\zetaT(t)\coloneqq\sum_{\boldsymbol{k}\in D_t} \prod_{v\in V(t)} \frac{1}{k_v^{n_v}},
	\]
	where $D_t$ is the set defined by:
	\[
	\left\{ (k_v)_{v\in V(t)}\in(\N^*)^{| V(t)|} \,\middle|\, \forall (v,u)\in  V(t)^2 , k_v<k_u \text{ if } u< v  \right\},
	\] 
    where for any $(u,v)\in V(t)^2$, $u<v$ stands for $u\leq v$ and $u\neq v$, with $u\leq v$ comes from reading the tree as a Hasse diagram where the root is the minimum of the poset.
	
	\medskip
	
	Now, let $\calPTxyconv$ be the subspace of $\calPTxy$ spanned by trees whose roots are decorated by $x$ and leaves are decorated by $y$. We define the linear map $\zetaTsh:\calPTxyconv \rightarrow  \R$ such that for any tree $t$ of $\calPTxyconv$:
    \[
    \zetaTsh(t)\coloneqq\int_{u\in\Delta_t} \prod_{v\in V(t)} g_{d_v}(t_v)\dt_v,
    \]
    where $\Delta_t\subseteq \interff{0 1}^{| V(t)|}$ stands for:
    \[ 
    \left\{ \left(t_{v_1},\dots, t_{v_{| V(t)|}}\right) \,\middle|\, \forall (i,j)\in\IEM{1}{| V(t)|}^2, t_{v_i}> t_{v_j} \text{ if } v_i < v_j  \right\},
    \] 
    where for any $(u,v)\in V(t)^2, u< v$ reading the tree as a Hasse diagram where the root is the minimum of the poset and 
    \begin{align*}
    	g_x:\left\lbrace\begin{array}{rcl}
    		\interoo{0 1} & \rightarrow & \R, \\
    		t & \mapsto & \frac{1}{t},  
    	\end{array} \right. &&
    	g_y:\left\lbrace\begin{array}{rcl}
    		\interoo{0 1} & \rightarrow & \R, \\
    		t & \mapsto & \frac{1}{1-t}.
    	\end{array} \right.
    \end{align*}
    
\end{defi}

\begin{Egs}
	We give some instances of computations for both zeta functions:
	\begin{multline*}
		\zeta^T\left(\labY{2}{2}{1}\right)= \sum_{\substack{0< k_2 < k_1 \\ 0< k_3 <k_1}} \frac{1}{k_1^2}\frac{1}{k_2^2}\frac{1}{k_3}, \quad
	\zetaTsh\left(\raisebox{-0.3\height}{\begin{tikzpicture}[line cap=round,line join=round,>=triangle 45,x=0.2cm,y=0.2cm, anchor=base, baseline]
		\draw (0,0) -- (0,1);
		\draw (0,1) -- (1,2);
		\draw (0,1) -- (-1,2);
		\draw (-1,2) -- (-1,3);
		\filldraw (0,0) circle (1pt);
		\filldraw (0,1) circle (1pt);
		\filldraw (1,2) circle (1pt);
		\filldraw (-1,2) circle (1pt);
		\filldraw (-1,3) circle (1pt);
		\draw (0,0) node[left] {\footnotesize{x}};
		\draw (0,1) node[left] {\footnotesize{x}};
		\draw (1,2) node[right] {\footnotesize{y}};
		\draw (-1,2) node[left] {\footnotesize{x}};
		\draw (-1,3) node[left] {\footnotesize{y}};
	\end{tikzpicture}}\right)= \int_{\substack{0<t_5<t_2<t_1<1 \\ 0<t_4<t_3<t_2}} \frac{1}{t_1t_2t_3(1-t_4)(1-t_5)} \\
	\quad \zetaTsh\left(\raisebox{-0.3\height}{\begin{tikzpicture}[line cap=round,line join=round,>=triangle 45,x=0.2cm,y=0.2cm, anchor=base, baseline]
		\draw (0,0) -- (0,4);
		\filldraw (0,0) circle (1pt);
		\filldraw (0,1) circle (1pt);
		\filldraw (0,2) circle (1pt);
		\filldraw (0,3) circle (1pt);
		\filldraw (0,4) circle (1pt);
		\draw (0,0) node[left] {\footnotesize{x}};
		\draw (0,1) node[left] {\footnotesize{x}};
		\draw (0,2) node[left] {\footnotesize{y}};
		\draw (0,3) node[left] {\footnotesize{x}};
		\draw (0,4) node[left] {\footnotesize{y}};
	\end{tikzpicture}}\right)= \int_{\substack{0<t_5<t_4<t_3 \\ t_3< t_2<t_1 <1}} \frac{1}{t_1t_2(1-t_3)t_4(1-t_5)},
	\quad \zeta^T\left(\raisebox{-0.3\height}{\begin{tikzpicture}[line cap=round,line join=round,>=triangle 45,x=0.2cm,y=0.2cm, anchor=base, baseline]
	\draw (0,0) -- (0,3);
	\filldraw (0,0) circle (1pt);
	\filldraw (0,1) circle (1pt);
	\filldraw (0,2) circle (1pt);
	\filldraw (0,3) circle (1pt);
	\draw (0,0) node[left] {\scriptsize{$2$}};
	\draw (0,1) node[left] {\scriptsize{$1$}};
	\draw (0,2) node[left] {\scriptsize{$3$}};
	\draw (0,3) node[left] {\scriptsize{$4$}};
	\end{tikzpicture}}\right) =  \sum_{0< k_4 < k_3 <k_2 <k_1} \frac{1}{k_1^2} \frac{1}{k_2} \frac{1}{k_3^3}\frac{1}{k_4^4}.
	\end{multline*}
\end{Egs}

We introduce $\iota:\Wcv \rightarrow \PTN$ a canonical embedding of words onto trees defined for any $w=w_1\dots w_k\in\Wcv$:

\begin{equation} \label{eq:iota_map}
	\iota\left( w_1\dots w_k \right)= \raisebox{-0.3\height}{\begin{tikzpicture}[line cap=round,line join=round,>=triangle 45,x=0.3cm,y=0.3cm, anchor=base, baseline]
	\draw (0,0) -- (0,1);
	\draw[dashed] (0,1) -- (0,2);
	\draw (0,2) -- (0,3);
	\filldraw (0,0) circle (1pt);
	\filldraw (0,1) circle (1pt);
	\filldraw (0,2) circle (1pt);
	\filldraw (0,3) circle (1pt);
	\draw (0,0) node[left] {\footnotesize{$w_1$}};
	\draw (0,1) node[left] {\footnotesize{$w_2$}};
	\draw (0,2) node[left] {\footnotesize{$w_{k-1}$}};
	\draw (0,3) node[left] {\footnotesize{$w_k$}};
\end{tikzpicture}} = B_+^{w_1}\circ B^{w_2}_+ \circ \ldots B_+^{w_{k-1}} \circ B_+^{w_k}(\emptyset).
\end{equation}

Doing so we have $\zeta = \zetaT \circ\iota.$
    Notice that these iterated series and integrals are indeed convergent for any trees in $\calPTNconv$ and $\calPTxyconv$ respectively. This was proved in the context of non-planar trees~\cite{clavier2020double} (and forests) but the planarity does not play any analytical role in the above definition.
    	
    Another result from~\cite{clavier2020double} that we will make an extensive use of requires some more definitions.
    \begin{defi}\label{defi:flat_map}
	We define the map $\flaten_1:\calPTN\rightarrow \W$ by $\flaten_1(\tdun{n})=(n)$ and for any ${\B_+^n(t_1\dots t_k)=t}$ with $n\in\N^*$ by:
	\[
	\flaten_1(t)=(n)\sqcup(\flaten_1(t_1)\cshuffle\cdots \cshuffle \flaten_1(t_k)).
	\]
	Further define the map $\flaten_0:\calPTxy\rightarrow \Wxy$ by $\flaten_0(\tdun{\epsilon})=(\epsilon)$ for $\epsilon\in\{x,y\}$ and for any $\B^\epsilon_+(t_1\dots t_k)=t$ with $\epsilon\in\{x,y\}$ by:
		\[
		\flaten_0(t)=(\epsilon)\sqcup(\flaten_0(t_1)\shuffle\cdots \shuffle \flaten_0(t_k)).
		\]
\end{defi}
    These flattening maps are called contracting and simple arborifications in~\cite{Ec92} as well as in \cite{Manchon_16}. The map $\flaten_0$ is called the Hairer-Kelly map in the context of rough paths theory, from the paper \cite{hairer2015geometric}. It is also related to the notion of topological ordering in data science (see \cite{kahn1962topological}).
    
    Notice that a similar definition of $\flaten_0$ can be made between $\calPTO$ and $\WO$ for any set $\Omega$. For such maps we do not indicate the dependance in $\Omega$ since their definition domain and image will always be explicitely stated.
    
    \medskip
    
    The next results relate \AZVs{} and \MZVs{} and will be used through this article.
    \begin{thm}[\cite{clavier2020double}] \label{thm:AZV_flaten}
     We have $\flaten_1\left(\calPTNconv\right)=\Wcv$ as well as $\flaten_0\left(\calPTxyconv\right)=\Wxycv$ and
     \begin{equation*}
      \zetaT=\zeta\circ\flaten_1,\qquad\zetaTsh=\zetash\circ\flaten_0.
     \end{equation*}
    \end{thm}
    
\subsection{Arborified zeta values and the (tri)dendriform (quasi-)shuffle}


We start with a simple yet important lemma.
\begin{Lemme} \label{lemma:flat_dend_morph}
 Let $\Omega$ be a set (resp. $(\Omega,+)$ a commutative semigroup). The map 
 \begin{equation*}
  \flaten_0:(\calPTO,<,>)\longrightarrow(\WO,\prec,\succ)\qquad(\text{\normalfont{ resp. } }\flaten_1:(\calPTO,<,>,\bullet)\longrightarrow(\WO,\prec,\succ,\cdot))
 \end{equation*}
is a morphism of dendriform algebras (resp. of tridendriform algebras).
\end{Lemme}
\begin{proof}
 We want to prove that for any pair of trees $(S,T)$ we have ${\flaten_0(S<T)=\flaten_0(S)\prec\flaten_0(T)}$ and ${\flaten_0(S>T)=\flaten_0(S)\succ\flaten_0(T)}$. We do this by induction on $|S|+|T|$.
 
 If $|S|+|T|=2$, then there are two elements $x$ and $y$ of $\Omega$ such that $S=\tdun{x}$ and $T=\tdun{y}$. Then by definition of $<$, $\flaten_0$, and $\prec$ we have
 \begin{equation*}
  \flaten_0(S<T)=\flaten_0\left(\tddeux{x}{y}\right)=xy=(x)\prec (y)=\flaten_0(\tdun{x})\prec\flaten_0(\tdun{y}).
 \end{equation*}
 The same computation gives also $\flaten_0(S>T)=\flaten_0(\tdun{x})\succ\flaten_0(\tdun{y})$.
 
 Now assume that for some $N\geq2$, the lemma holds for any pair of planar rooted trees $(S,T)$ with $|S|+|T|\leq N$. So, in particular $\flaten_0(S\shuffle^T T)=\flaten_0(S)\shuffle\flaten_0(T)$ for any such pair $(S,T)$.
 
 Let $(S,T)$ be a pair of trees such that $|S|+|T|=N+1$. As before, write $S=B_+^x(S_1\cdots S_k)$ and $T=B_+^y(T_1\cdots T_\ell)$ for some $x$ and $y$ in $\Omega$, $k$ and $\ell$ greater or equal to zero and trees $S_i$ and $T_j$. This time, in order to not introduce non-standard notations for words, we will distinguish the case $k=0$ and $\ell=0$ from the general case.
 \begin{itemize}
  \item If $k=0$ then by definition of $<$, $\flaten_0$, and $\prec$:
  \begin{align*}
   \flaten_0(S<T) & =\flaten_0(\tdun{x}<T)=\flaten_0(B_+^x(T))=(x)\sqcup\flaten_0(T)=(x)\prec\flaten_0(T) \\
   & =\flaten_0(S)\prec\flaten_0(T).
  \end{align*}
  \item If $k\geq1$ then on the one hand
  \begin{align*}
   \flaten_0(S<T) & = \flaten_0\left(B_+^x(S_1\cdots(S_k\shuffle^T T))\right) \tag{by definition of <} \\
   & = (x)\sqcup(\flaten_0(S_1)\shuffle\cdots\shuffle\flaten_0(S_k\shuffle^T T))\tag{by definition of $\flaten_0$} \\
   & = (x)\sqcup(\flaten_0(S_1)\shuffle\cdots\shuffle\flaten_0(S_k)\shuffle \flaten_0(T))
  \end{align*}
  by the induction hypothesis and associativity of $\shuffle$.
  
  On the other hand we also have by definition of $\flaten_0$
  \begin{align*}
   \flaten_0(S)\prec\flaten_0(T) & = \left((x)\sqcup\flaten_0(S_1)\shuffle\cdots\shuffle\flaten_0(S_k)\right)\prec\flaten_0(T) \\
   & = (x)\sqcup(\flaten_0(S_1)\shuffle\cdots\shuffle\flaten_0(S_k)\shuffle \flaten_0(T))
  \end{align*}
  by definition of $\prec$ and associativity of $\shuffle$.
 \end{itemize}
 Thus, by induction $\flaten_0(S<T)=\flaten_0(S)\prec\flaten_0(T)$ for any pair of trees.
 
 For the second relation, we again need to distinguish two cases.
 \begin{itemize}
  \item If $\ell=0$ then by definition of $>$, $\flaten_0$, and $\succ$:
  \begin{align*}
   \flaten_0(S>T) & =\flaten_0(S>\tdun{y})=\flaten_0(B_+^y(S))=(y)\sqcup\flaten_0(S)=\flaten_0(S)\succ (y) \\
   & = \flaten_0(S)\succ \flaten_0(T).
  \end{align*}
  \item If $\ell\geq1$ then on the one hand
  \begin{align*}
   \flaten_0(S>T) & = \flaten_0\left(B_+^y((S\shuffle^T T_1)\cdots T_\ell)\right) \tag{by definition of >} \\
   & = (y)\sqcup(\flaten_0(S\shuffle^T T_1)\shuffle\cdots\shuffle\flaten_0(T_\ell))\tag{by definition of $\flaten_0$} \\
   & = (y)\sqcup(\flaten_0(S)\shuffle\flaten_0(T_1)\shuffle\cdots\shuffle\flaten_0(T_\ell))
  \end{align*}
  by the induction hypothesis and associativity of $\shuffle$.

 And on the other hand
  \begin{align*}
   \flaten_0(S)\succ\flaten_0(T) & = \flaten_0(S)\succ\left((y)\sqcup\flaten_0(T_1)\shuffle\cdots\shuffle\flaten_0(T_\ell)\right)\tag{by definition of $\flaten_0$} \\
   & = (y)\sqcup(\flaten_0(S)\shuffle\flaten_0(T_1)\shuffle\cdots\shuffle\flaten_0(T_\ell)).
  \end{align*}
  by definition of $\succ$ and associativity of $\shuffle$.
\end{itemize}
  Thus, by induction we also have $\flaten_0(S>T)=\flaten_0(S)\succ\flaten_0(T)$ for any pair of trees. \\
  
  For the tridendriform case, exactly the same proof shows that
  \begin{equation*}
   \flaten_1(S>T)=\flaten_1(S)\succ\flaten_1(T),\quad\text{and}\quad\flaten_1(S<T)=\flaten_1(S)\prec\flaten_1(T).
  \end{equation*}
  We prove by induction of $|S|+|T|$ that $\flaten_1(S\bullet T)=\flaten_1(S)\cdot\flaten_1(T)$.
  
  If $|S|+|T|=2$, then $S=\tdun{x}$ and $T=\tdun{y}$ for some $(x,y)\in\Omega^2$. We then have
   \begin{equation*}
    \flaten_1(S\bullet T)=\flaten_1(\tdun{x+y}\hspace{0.6cm})=(x+y)=(x)\cdot (y)=\flaten_1(S)\cdot\flaten_1(T)
   \end{equation*}
   by definition of $\bullet$, of $\flaten_1$, and of $\cdot$. Notice that in the computation above, $+$ is the commutative law of $\Omega$, and $x+y$ in the second to last term is one word of length one.
   
   Assume now that the result holds for any pair of trees $(S,T)$ such that $|S|+|T|\leq N$ for some $N\geq2$. Notice that this implies that $\flaten_1(S\cshuffle^T T)=\flaten_1(S)\cshuffle\flaten_1(T)$ for any such pair $(S,T)\in\calPTO$. Then let $(S,T)$ be any pair of planar rooted trees such that $|S|+|T|=N+1$. Then we can write $S=B_+^x(S_1\cdots S_k)$ and $T=B_+^y(T_1\cdots T_\ell)$ for some $x$ and $y$ in $\Omega$, $k$ and $\ell$ greater or equal to zero and planar rooted trees $S_i$ and $T_j$. Then we have on the one hand
   \begin{align*}
    \flaten_1(S\bullet T) & = \flaten_1(B_+^{x+y}(S_1\cdots(S_k\cshuffle^T T_1)\cdots T_\ell)) \quad\text{(by definition of $\bullet$)}, \\
    & = (x+y)\sqcup\left(\flaten_1(S_1)\cshuffle\cdots\cshuffle\flaten_1(S_k\cshuffle^T T_1)\cshuffle\cdots\cshuffle\flaten_1(T_\ell)\right)\quad\text{(by definition of $\flaten_1$)},\\
    & = (x+y)\sqcup\left(\flaten_1(S_1)\cshuffle\cdots\cshuffle\flaten_1(S_k)\cshuffle\flaten_1(T_1)\cshuffle\cdots\cshuffle\flaten_1(T_\ell)\right)
   \end{align*}
   by the induction hypothesis and associativity of $\cshuffle$. And on the other hand, by definition of $\flaten_1$:
   \begin{align*}
    \flaten_1(S)\cdot\flaten_1(T) & = \left((x)\sqcup\flaten_1(S_1)\cshuffle\cdots\cshuffle\flaten_1(S_k)\right)\cdot\left((y)\sqcup\flaten_1(T_1)\cshuffle\cdots\cshuffle\flaten_1(T_\ell)\right), \\
    & = (x+y)\sqcup(\flaten_1(S_1)\cshuffle\cdots\cshuffle\flaten_1(S_k)\cshuffle\flaten_1(T_1)\cshuffle\cdots\cshuffle\flaten_1(T_\ell))
   \end{align*}
   by definition of $\cdot$ and associativity of $\cshuffle$. Notice that in this case we do not need to separate the cases $k=0$ and $\ell=0$ by using the aformentioned convention $\emptyset\cshuffle^T T=T\cshuffle^T\emptyset=T$ and $\flaten_1(\emptyset)=\emptyset$. Thus we have proven the lemma by induction.
\end{proof}
Notice this lemma implies for any trees $S$ and $T$ (resp. for $(\Omega,+)$ a commutative semigroup) $\flaten_0(S\shuffle^T T)=\flaten_0(S)\shuffle\flaten_0(T)$ (resp. ${\flaten_1(S\cshuffle^T T)=\flaten_1(S)\cshuffle\flaten_1(T)}$. Now we can prove the main result of this subsection:
\begin{thm} \label{thm:azv_shuffle_stuffle}
 The map $\zetaTsh:\calPTxyconv\longrightarrow\R$ is an algebra morphism for the dendriform shuffle product $\shuffle^T$ and the map $\zetaT:\calPTNconv\longrightarrow\R$ is an algebra morphism for the tridendriform quasi-shuffle product $\cshuffle^T$.
\end{thm}
\begin{proof}
 For any convergent trees $S$ and $T$ we have
 \begin{align*}
  \zetaTsh(S\shuffle^T T) & = \zetash\circ\flaten_0(S\shuffle^T T) \quad\text{by theorem }\ref{thm:AZV_flaten} \\
  & = \zetash\left(\flaten_0(S)\shuffle\flaten_0(T)\right) \quad\text{by lemma }\ref{lemma:flat_dend_morph} \\
  & = \zetash\circ\flaten_0(S)\cdot\zetash\circ\flaten_0(T) \quad\text{by equation }\eqref{eq:shuffle_stuffle_zeta} \\
  & = \zetaTsh(S)\zetaTsh(T) \quad\text{by theorem }\ref{thm:AZV_flaten}.
 \end{align*}
 Exactly the same argument implies that $\zetaT$ is an algebra morphism for $\cshuffle^T$.
\end{proof}
	\begin{Rq}
     Notice that $\zetash$ and $\zetaT$ can be defined via arborification of Rota-Baxter maps, see the second author's work~\cite{clavier2020double} for details. Using these methods and those of~\cite{clavier2020locality}, one could prove theorem~\ref{thm:azv_shuffle_stuffle} without reference to flattening maps. We do not chose this solution to avoid the introduction of new structures and postpone to future work the detail study of the links between arborification of Rota-Baxter maps and (tri)dendriform structures on planar rooted trees.
    \end{Rq}


\section{Arborified integral representation and Hoffman relation} \label{sec:Hoffman}

\subsection{Process trees}

We start by recalling the construction of the third author~\cite{Fan25}.
\begin{defi} \label{def:min_incomp_pairs}
  Let $T\in\calPTN$ be a non-empty planar tree which is not a ladder tree. Then we can uniquely write 
  \begin{equation*}
   T=B_+^{n_1}\circ\cdots\circ B_+^{n_p}(T_1\cdots T_k)
  \end{equation*}
  for some non-empty planar trees $T_1,\cdots,T_k$ and $k\geq2$. Let $a$ be the root of $T_1$ and $b$ be the root of $T_2$. Then we call $(a,b)$ the \emph{minimal incomparable pair} of $T$.
  And so $T$ can be written
  \begin{equation*}
   T=B_+^{n_1}\circ\cdots\circ B_+^{n_p}(B_+^{n_a}(F_a)B_+^{n_b}(F_b)F)
  \end{equation*}
  where $F_a$, $F_b$ and $F$ are planar forests, eventually empty. Then we set as in~\cite{Fan25}
  \begin{align*}
   T_b^a & := B_+^{n_1}\circ\cdots\circ B_+^{n_p}\left(B_+^{n_a}(B_+^{n_b}(F_b)F_a)F\right), \\
   T_a^b & := B_+^{n_1}\circ\cdots\circ B_+^{n_p}\left(B_+^{n_b}(B_+^{n_a}(F_a)F_b)F\right), \\
   T_{a+b} & :=B_+^{n_1}\circ\cdots\circ B_+^{n_p}\left(B_+^{n_a+n_b}(F_aF_b)F\right).
  \end{align*}
  \end{defi}
  
 

Let us define maps that will allow us to iteratively define the process tree.
 \begin{defi}
     Let $f_1,f_2,f_3:\calPTN\longrightarrow\calPTN$ be three linear maps defined by induction with
     \begin{equation*}
         f_1(\tdun{n})=f_2(\tdun{n})=f_3(\tdun{n})=\tdun{n}
     \end{equation*}
     for all $n\in\N^*$, and for $T=B_+^n(T_1\cdots T_k)$ with $T_i=B_+^{a_i}(F_i)$:
     \begin{gather*}
         f_1(T)  = \begin{cases} 
         B_+^n(f_1(T_1)) &\text{if }k=1 \\
         B_+^n\left(B_+^{a_1}(T_2F_1)\right) &\text{if } k=2 \\ 
         B_+^n\left(B_+^{a_1}(T_2F_1)T_3\cdots T_k\right)&\text{if }k\geq3
         \end{cases}, \quad
         f_2(T)  = \begin{cases} 
         B_+^n(f_2(T_1)) &\text{if }k=1 \\
         B_+^n\left(B_+^{a_2}(T_1F_2)\right) &\text{if }k=2 \\
         B_+^n\left(B_+^{a_2}(T_1F_2)T_3\cdots T_k\right)&\text{if }k\geq3
         \end{cases}, \\
         f_3(T)  =  \begin{cases} 
         B_+^n(f_3(T_1))&\text{if }k=1 \\
         B_+^n\left(B_+^{a_1+a_2}(F_1F_2) \right) &\text{if } k=2 \\
         B_+^n\left(B_+^{a_1+a_2}(F_1F_2)T_3\cdots T_k\right)&\text{if }k\geq3.
         \end{cases}
     \end{gather*}
     We extend $f_i$ by linearity to a map of $\calPTN$.
 \end{defi}

\begin{Lemme} \label{lem:finite_lin_number}
  For every tree $T\in\calPTN$, there exists $M\in\N$ such that,
  for every $m\geq M$ and every $g_1,\ldots,g_m\in \{f_1,f_2,f_3\}$, the tree $g_1\circ\cdots\circ g_m(T)$ is a ladder tree.
\end{Lemme}
 \begin{proof}
     We prove the result by induction on $|V(T)|$. If $|V(T)|=1$ the result trivially holds. Assume that there exists $N\in\N^*$ such that the result holds for all trees with $|V(T)|\leq N$. Let $T$ be a tree with $N+1$ vertices and let $k$ be the number of direct descendants of its root. Then we write $T=B_+^n(T_1\cdots T_k)$ where for $i\in\IEM{1}{k}$, $T_i=B_+^{a_i}(F_i)$ with $a_i\in\N^*$ and $F_i$ is a forest. 

     \begin{itemize}
     	\item If $k=1$ then for any $j\in\IEM{1}{3}, f_j(T)=B_+^n(f_j(T_1))$ . Applying the induction hypothesis on $T_1$, there exists $M\in\N$ such that, for every $m\geq M$ and every $g_1,\ldots,g_m\in \{f_1,f_2,f_3\}$, $g_1\circ\cdots\circ g_m(T_1)$ is a ladder tree. Then, as $g_1\circ\cdots\circ g_m(T)=B_+^n(g_1\circ\cdots\circ g_m(T_1))$ which is a ladder tree, the lemma holds in this case. 
     	
     	
     	\item If $k \geq 2$, for any $i \in \{1,2,3\}$ note that $f_i(T)$ is a tree whose root has strictly less direct descendants than the one of $T$. Consequently, for any $g_1, \ldots, g_k \in \{f_1,f_2,f_3\}$ we have $g_1 \circ \cdots \circ g_k(T) = B_+^n \circ B_+^r(\widetilde{T})$ for some $r \in \N^*$ depending on the $a_i$'s and $g_i$'s, and $\widetilde{T} \in \calPTN$ with $|\widetilde{T}| < |T|$ depending of the $T_i$'s and $g_i$'s. Since there are only finitely many choices of $g_1, ..., g_k$, there are only finitely many possible $\widetilde{T}$. By induction, for each possible $\widetilde{T}$ there exists $M_{\widetilde{T}}$ such that $\widetilde{T}$ is a ladder tree after applying any $M_{\widetilde{T}}$ compositions of elements from $\{f_1,f_2,f_3\}$. Let $M = \mathrm{max}_{\widetilde{T}} M_{\widetilde{T}}$. So for $T$, any composition of $M+k$ such elements also gives a ladder tree.
     \end{itemize}
     
     Thus we have proven the lemma by induction.
 \end{proof}
 

 We can now define the counter we will use later on.
\begin{defi} \label{defi:lin}
    Let $F = \{f_1, f_2, f_3\}$ be the alphabet, and let $g = g_1\sqcup\cdots\sqcup g_n$ be a word
    in $F^\star$. We set $\lin(T)=0$ if $T$ is a ladder tree and
    \[
    \lin(T)\coloneqq \mathrm{max}\left\{n\in\N \,\middle|\, \exists g\in F^\star \text{ s.t. } g_1\circ\cdots\circ g_n(T) \text{ is not a ladder tree} \right\} + 1.
    \]
\end{defi}
 Notice that by lemma \ref{lem:finite_lin_number}, $\lin(T)$ is finite for every tree $T$ and $\lin(T)=0$ if, and only if, $T$ is a ladder tree. 
 The next simple lemma is an essential stepping stone for our construction.
 \begin{Lemme} \label{lem:lin_decreases}
     Let $T\in\calPTN$ be a non-ladder tree with minimal incomparable pair $(a,b)$. Then 
     \begin{equation*}
         \lin\left(T_b^a\right)<\lin(T),\quad \lin\left(T_a^b\right)<\lin(T),\quad\lin\left(T_{a+b}\right)<\lin(T).
     \end{equation*}
 \end{Lemme}
 \begin{proof}
    Let $T\in\calPTN$ be a non-ladder tree with minimal incomparable pair $(a,b)$. Put ${\lin(T)=N}$. Assume $\lin(T^a_b)=k\in\N$. Then there exists 
    $g_1\sqcup\dots\sqcup g_{k-1}\in F^\star$ such that $g_1\circ\cdots\circ g_{k-1}(T^a_b)$ is not a ladder tree. Since $T^a_b=f_1(T)$ we have a word 
    $g_1\sqcup\dots\sqcup g_{k-1}\sqcup f_1\in F^\star$ such that $g_1\circ\cdots\circ g_{k-1}\circ f_1(T)$ is not a ladder tree. Thus $\lin(T)\geq k+1>k=\lin(T^a_b)$. 
    The other inequalities are proven in the same way, replacing $f_1$ by $f_2$ and $f_3$.
 \end{proof}

  This lemma allows us to define recursively the process tree of a tree, a crucial construction of the third author's work~\cite{Fan25}. Note that a process tree is a tree decorated by trees.


\begin{defi} \label{defi:process_tree} 
  Let $T\in\calPTN$ be a planar rooted tree. Its \emph{process tree $\pr(T)$}
  is a planar rooted tree whose vertices are decorated by elements from $\calPTN$ defined recursively on $\lin(T)$.
  \begin{itemize}
   \item If $\lin(T)=0$, then $T$ is a ladder tree (eventually empty) and we set $\pr(T):=\tdun{T}$.
   \item Assume that for $N\in\N$, $\pr(S)$ has been defined for all $S\in\calPTN$ with $\lin(S)\leq N$. Let $T\in\calPTN$ be a planar rooted tree with $\lin(T)=N+1$. Then $T$ is not a ladder tree and
   let $(a,b)$ be its minimal incomparable pair. We set
   $$\pr(T) := B_+^{T}\left(\pr(T_{a+b})\, \pr(T_{b}^{a})\, \pr(T_{a}^{b})\right).$$
  \end{itemize}
  For any tree $T\in\calPTN$, we write $\delta_{\pr(T)}:V(\pr(T))\longrightarrow\calPTN$ the decoration map of $\pr(T)$.
\end{defi}
Notice that the above definition makes sense by lemma~\ref{lem:lin_decreases}.

We will need one technical result regarding the process tree.


\begin{Lemme} \label{lem:pr_B_plus}
 For any tree $T\in\calPTN$ and any $n\in\N^*$ we have
 \begin{equation*}
  (\pr(B_+^n(T)),\delta_{\pr(B_+^n(T))}) = (\pr(T),B_+^n\circ\delta_{\pr(T)}).
 \end{equation*}
\end{Lemme}
\begin{proof}
 This is easily proved by induction on $\lin(T)$ thanks to the following facts which are direct consequences of the definitions of $T_a^b$, $T_b^a$, and $T_{a+b}$: ${B_+^n(T)_{a+b}=B_+^n(T_{a+b})}$, ${B_+^n(T)_a^b=B_+^n(T_a^b)}$, and ${B_+^n(T)_b^a=B_+^n(T_b^a)}$.
\end{proof}

\subsection{Branched binarisation map}

The next map is a planar version of the branched binarisation map~\cite{clavier2020double}.
\begin{defi} \label{defi:old_bin_map}
 Let $\fraks^{PN}:\calPTN\longrightarrow\calPTxy$ be the linear map defined inductively by ${\fraks^{PN}(\emptyset)=\emptyset}$ and
 \begin{equation*}
  \fraks^{PN}(B_+^n(F))=\underbrace{B_+^x \circ \cdots \circ B_+^x}_{n-1\text{ times}}\circ B_+^y(\fraks^{PN}(F)), \qquad\fraks^{PN}(T_1\cdots T_k)=\fraks^{PN}(T_1)\cdots\fraks^{PN}(T_k).
 \end{equation*}
 It is clear that $\fraks^{PN}\left(\calPTNconv\right) \subseteq \calPTxyconv$, but this simple fact will play an important role in the rest of the paper.
 
\end{defi}
Notice that $\fraks^{PN}$ can be defined using a universal property of $\calPTN$, namely that $\calPTN$ is the initial object in the category of $\N^*$-operated algebras. See~\cite{guo2009operated,kreimer2013renormalization} as well as~\cite{clavier2020locality} for an elementary proof.
Since we are not dealing with these objects in this paper we chose to use the more mundane definition above, which is stricly equivalent. On ladder trees $\fraks^{PN}$ coincides with the binarisation map $\fraks:\calW_{\N^*}\longrightarrow\calW_{\{x,y\}}$ defined in equation~\eqref{eq:bin_map_words}. 

\begin{Egs} \label{ex:sPN}
A simple computation gives
\begin{equation*}
 \fraks^{PN} \left(\begin{xy}
      {(-4,2) \ar @{{*}-{*}} (0,-2)},
      {(4,2) \ar @{{*}-{*}} (0,-2)},
      (-2,2)*{a}, (2,-2)*{b}, (6,2)*{c},
    \end{xy}\right) = (B_+^{x})^{b-1}\circ B_+^{y}\left((B_+^{x})^{a-1}\circ B_+^{y}\left(\emptyset\right)(B_+^{x})^{c-1}\circ B_+^{y}\left(\emptyset\right)\right) 
    = \begin{xy}
      {(-4,10) \ar @{{*}-{*}} (-4,6)},
      {(-4,6) \ar @{{*}.{*}} (-4,2)},
      (-2,10)*{y},
      (-2,6)*{x},
      (-2,2)*{x},
      {(-6.5,2) \ar @{{}-{}} (-5.5,2)},
      {(-6,4) \ar @{{}-{}} (-6,2)},
      {(-6.5,10) \ar @{{}-{}} (-5.5,10)},
      {(-6,8) \ar @{{}-{}} (-6,10)},
      (-6,6)*{a},
      {(4,10) \ar @{{*}-{*}} (4,6)},
      {(4,6) \ar @{{*}.{*}} (4,2)},
      (6,10)*{y},
      (6,6)*{x},
      (6,2)*{x},
      {(1.5,2) \ar @{{}-{}} (2.5,2)},
      {(2,4) \ar @{{}-{}} (2,2)},
      {(1.5,10) \ar @{{}-{}} (2.5,10)},
      {(2,8) \ar @{{}-{}} (2,10)},
      (2,6)*{c},
      {(0,-2) \ar @{{*}-{*}} (0,-6)},
      {(0,-6) \ar @{{*}.{*}} (0,-10)},
      (2,-2)*{y},
      (2,-6)*{x},
      (2,-10)*{x},
      {(-2.5,-10) \ar @{{}-{}} (-1.5,-10)},
      {(-2,-8) \ar @{{}-{}} (-2,-10)},
      {(-2.5,-2) \ar @{{}-{}} (-1.5,-2)},
      {(-2,-4) \ar @{{}-{}} (-2,-2)},
      (-2,-6)*{b},
      {(-4,2) \ar @{{*}-{*}} (0,-2)},
      {(4,2) \ar @{{*}-{*}} (0,-2)},
    \end{xy}.
\end{equation*}

\end{Egs}

We can now define the binarisation map of the third author's work~\cite{Fan25}. 
\begin{defi} \label{defi:error_tree} 
%
  Let $T\in\PTN$ be a planar rooted tree. We define its \emph{error term $T^e$} as follows: if $T$ is a ladder tree (eventually empty), then $T^e=0$. Otherwise, let $(a,b)$ be its minimal incomparable pair, write $T=B_+^{n_1}\circ\cdots\circ B_+^{n_p}(B_+^{n_a}(F_a)B_+^{n_b}(F_b)F)$ as before and set
  \begin{align*}
   T^e = & B_+^{n_1}\circ\cdots\circ B_+^{n_p}(B_+^{n_a+n_b}(F_aF_b)F) \\
   & - \sum_{i=1}^{n_a-1}\binom{n_a-i+n_b-1}{n_b-1}B_+^{n_1}\circ\cdots\circ B_+^{n_p}\left(B_+^{n_a+n_b-i}(B_+^{i}(F_a)F_b)F\right) \\
   & - \sum_{i=1}^{n_b-1}\binom{n_b-i+n_a-1}{n_a-1}B_+^{n_1}\circ\cdots\circ B_+^{n_p}\left(B_+^{n_a+n_b-i}(B_+^{i}(F_b)F_a)F\right)
  \end{align*}
  where the sum $\sum_{i=1}^{n_a-1}$ (resp. $\sum_{i=1}^{n_b-1}$) is taken to be zero if $n_a=1$ (resp. if $n_b=1$).
  
We extend the map $T\mapsto T^e$ to a map on $\calPTN$ by linearity.
  

  Recall that for $T\in\calPTN$, $\pr(T)$ is a tree decorated by elements of $\calPTN$ and that $\delta_{\pr(T)}$ is its decoration map. Then, we define the linear map $\phi : \calPTN\longrightarrow\calPTN$
by
  \[
  \phi(T) = T + \sum_{v\in V(\pr(T))}\left(\delta_{\pr(T)}(v)\right)^e
  \]
  The map $\fraks^{PT} : \calPTN\longrightarrow\calPTxy$
is defined by
  \[
  \fraks^{PT}(T) = \fraks^{PN}\circ \phi(T).
  \]
\end{defi}

\begin{Eg} \label{ex:phi_sPT}
 Setting $T=\begin{xy}
      {(-4,2) \ar @{{*}-{*}} (0,-2)},
      {(4,2) \ar @{{*}-{*}} (0,-2)},
      (-2,2)*{a}, (2,-2)*{b}, (6,2)*{c},
    \end{xy}$ as in example~\ref{ex:sPN}, we have $\pr(T)=\begin{xy}
      {(-4,2) \ar @{{*}-{*}} (0,-2)},
      {(0,2) \ar @{{*}-{*}} (0,-2)},
      {(4,2) \ar @{{*}-{*}} (0,-2)},
      (-2,3.5)*{\ell_1}, (2,3.5)*{\ell_2}, (2.5,-2)*{T}, (6,3.5)*{\ell_3},
    \end{xy}$ where $\ell_1$, $\ell_2$, and $\ell_3$ are ladder trees given by 
$\ell_1=\iota(b(a+c))$, $\ell_2=\iota(bac)$, and $\ell_3=\iota(bca)$, $\iota$ being the canonical embedding given by equation~\eqref{eq:iota_map}. Then, since $\ell_1^e=\ell_2^e=\ell_3^e=0$ we have $\phi(T)=T+T^e$. By definition of the error term:
    \begin{equation*}
     \phi(T)=\begin{xy}
      {(-4,2) \ar @{{*}-{*}} (0,-2)},
      {(4,2) \ar @{{*}-{*}} (0,-2)},
      (-2,2)*{a}, (2,-2)*{b}, (6,2)*{c},
    \end{xy}
    +\begin{xy}
      {(0,2) \ar @{{*}-{*}} (0,-2)},
      (5,2)*{a+c}, (2,-2)*{b},
    \end{xy}
    -\sum_{i=1}^{a-1}\binom{a-i+c-1}{c-1}
    \raisebox{-0.2\height}{
    \begin{xy}
      {(0,2) \ar @{{*}-{*}} (0,-2)},
      {(0,6) \ar @{{*}-{*}} (0,2)},
      (8,2)*{a+c-i}, (2,-2)*{b}, (2,6)*{i}
    \end{xy}}
    -\sum_{i=1}^{c-1}\binom{c-i+a-1}{a-1} 
    \raisebox{-0.2\height}{
    \begin{xy}
      {(0,2) \ar @{{*}-{*}} (0,-2)},
      {(0,6) \ar @{{*}-{*}} (0,2)},
      (8,2)*{a+c-i}, (2,-2)*{b}, (2,6)*{i}
    \end{xy}}.
    \end{equation*}
 Using definition~\ref{defi:old_bin_map}, one then has $\fraksPT(T)$. Taking for simplicity $a=b=c=2$ we find
 \begin{equation*}
 \fraksPT(T)=
  \begin{xy}
      {(-4,6) \ar @{{*}-{*}} (-4,2)},
      (-2,6)*{y},
      (-2,2)*{x},
      {(4,6) \ar @{{*}-{*}} (4,2)},
      (6,6)*{y},
      (6,2)*{x},
      {(0,-2) \ar @{{*}-{*}} (0,-6)},
      (2,-2)*{y},
      (2,-6)*{x},
      {(-4,2) \ar @{{*}-{*}} (0,-2)},
      {(4,2) \ar @{{*}-{*}} (0,-2)},
    \end{xy}
    +   \begin{xy}
      {(0,10) \ar @{{*}-{*}} (0,6)},
      {(0,6) \ar @{{*}-{*}} (0,2)},
      {(0,2) \ar @{{*}-{*}} (0,-2)},
      (2,10)*{y},
      (2,6)*{x},
      (2,2)*{x},
      {(0,-2) \ar @{{*}-{*}} (0,-6)},
      {(0,-6) \ar @{{*}-{*}} (0,-10)},
      (2,-2)*{x},
      (2,-6)*{y},
      (2,-10)*{x},
    \end{xy}
    -4~\begin{xy}
      {(0,10) \ar @{{*}-{*}} (0,6)},
      {(0,6) \ar @{{*}-{*}} (0,2)},
      {(0,2) \ar @{{*}-{*}} (0,-2)},
      (2,10)*{y},
      (2,6)*{y},
      (2,2)*{x},
      {(0,-2) \ar @{{*}-{*}} (0,-6)},
      {(0,-6) \ar @{{*}-{*}} (0,-10)},
      (2,-2)*{x},
      (2,-6)*{y},
      (2,-10)*{x},
    \end{xy}
 \end{equation*}
\end{Eg}

The following theorem is one of the main results from the work of the third author~\cite{Fan25}.
\begin{thm}\label{thm F.} 
  The following diagram is commutative.
  \begin{equation} \label{eq:Ku-Yu_main_result}
    \begin{tikzcd}
   & \calPTN \arrow[rd, "\fraks^{PN}"] & \\
  \calPTN  \arrow[rr, "\fraks^{PT}"] \arrow[ru, "\phi"] \arrow[d, "\flaten_1"'] & & \calPTxy  \arrow[d, "\flaten_0"] \\
  \calW_{\N^*} \arrow[rr, "\fraks"]               &  & \calW_{\{x,y\}}            
  \end{tikzcd}
  \end{equation}
\end{thm}
Specialising diagram~\ref{eq:Ku-Yu_main_result} to convergent trees, we obtain the simple yet important corollary.
\begin{Cor} \label{coro:arbo_Konts}
The following diagram is commutative:
 \begin{equation} \label{eq:conv_bin}
  \begin{tikzcd}
\calPTNconv \arrow[r, "\fraks^{PT}"] \arrow[d, "\flaten_1"'] & \calPTxyconv \arrow[d, "\flaten_0"] \\
\Wcv \arrow[r, "\fraks"]            & \Wxycv               
\end{tikzcd}
 \end{equation}
 Furthermore, the \AZVs{} of definition~\ref{def:azvs} admits an interated integral representation through
 the branched binarisation map $\fraks^{PT}$:
\begin{equation} \label{eq:int_rep_AZVs}
 \forall T\in\calPTNconv,\quad\zeta^T(T)=\zeta^T_\shuffle\circ\fraks^{PT}(T).
\end{equation}
\end{Cor}
\begin{proof}
  Let $T\in\calPTN$ and write $\phi(T)=\sum_{i\in I}c_i t_i$, where $I$ is a non-empty finite set, the $t_i$s are some planar rooted trees and the $c_i$s are some coefficients. Let $r_T$ be the root of $T$ and for all $i\in I$, $r_i$ be the root of $t_i$. Then by definition of $\phi$, for all $i\in I, d_{t_i}(r_i)=d_T(r_T)$ (i.e. the roots of each of the $t_i$ have the same decoration than the root of $T$).
 
 In particular, if $T$ is convergent, then so is $\phi(T)$. 
 Copying the proof of a lemma for non-planar trees given by P.~Clavier~\cite[Lemma A.3]{clavier2020double}, one can show $\fraksPT(\calPTNconv)\subseteq\calPTxyconv$ since planarity does not play a role for convergency.
  Moreover, it is easy to see, and proven in~\cite[Lemma 3.24]{clavier2020double} that $\flaten_1(\calPTNconv)\subseteq\Wcv$ and $\flaten_0(\calPTxyconv)\subseteq\Wxycv$. Furthermore $\fraks(\Wcv)\subseteq\Wxycv$ by the iterated integral representation of \MZVs{}. Thus specialising the square part of diagram~\eqref{eq:Ku-Yu_main_result} to $\calPTNconv$ gives diagram~\eqref{eq:conv_bin}.
 
 Let us now prove equation~\eqref{eq:int_rep_AZVs}. For any $T\in\calPTNconv$ we have
 \begin{align*}
  \zeta^T(T) & = \zeta\circ\flaten_1(T)\quad\text{ by theorem \ref{thm:AZV_flaten}} \\
  & = \zeta_\shuffle\circ\fraks\circ\flaten_1(T)\quad\text{ by relation \eqref{eq:Kontsevich}} \\
  & = \zeta_\shuffle\circ\flaten_0\circ\fraks^{PT}(T)\quad\text{ by diagram \eqref{eq:conv_bin}} \\
  & = \zeta^T_\shuffle\circ\fraks^{PT}(T)\quad\text{ by theorem \ref{thm:AZV_flaten}}.  \qedhere
 \end{align*}
\end{proof}
Thus the binarisation map of~\cite{Fan25} directly relates \AZVs{} without the need to use conical sums~\cite{clavier2024generalisations}. This is an important step forward toward a fully satisfactory generalisation of \MZVs{} to rooted trees. In the next subsection we show that we also have the last missing piece of this generalisation, namely an arborified version of Hoffman's regularization relation \eqref{eq:Hoffman}. For now, we finish this subsection with an example of relation given by corollary~\ref{coro:arbo_Konts}.
\begin{Eg}
 Using example~\ref{ex:phi_sPT} we find
 \begin{equation*}
  \zeta^T \left(\begin{xy}
      {(-4,2) \ar @{{*}-{*}} (0,-2)},
      {(4,2) \ar @{{*}-{*}} (0,-2)},
      (-2,2)*{2}, (2,-2)*{2}, (6,2)*{2},
    \end{xy}\right)
    = \zetaTsh\left(\begin{xy}
      {(-4,6) \ar @{{*}-{*}} (-4,2)},
      (-2,6)*{y},
      (-2,2)*{x},
      {(4,6) \ar @{{*}-{*}} (4,2)},
      (6,6)*{y},
      (6,2)*{x},
      {(0,-2) \ar @{{*}-{*}} (0,-6)},
      (2,-2)*{y},
      (2,-6)*{x},
      {(-4,2) \ar @{{*}-{*}} (0,-2)},
      {(4,2) \ar @{{*}-{*}} (0,-2)},
    \end{xy}\right)
    +\zetaTsh\left(~   \begin{xy}
      {(0,10) \ar @{{*}-{*}} (0,6)},
      {(0,6) \ar @{{*}-{*}} (0,2)},
      {(0,2) \ar @{{*}-{*}} (0,-2)},
      (2,10)*{y},
      (2,6)*{x},
      (2,2)*{x},
      {(0,-2) \ar @{{*}-{*}} (0,-6)},
      {(0,-6) \ar @{{*}-{*}} (0,-10)},
      (2,-2)*{x},
      (2,-6)*{y},
      (2,-10)*{x},
    \end{xy}\right)
    -4~\zetaTsh\left(~ \begin{xy}
      {(0,10) \ar @{{*}-{*}} (0,6)},
      {(0,6) \ar @{{*}-{*}} (0,2)},
      {(0,2) \ar @{{*}-{*}} (0,-2)},
      (2,10)*{y},
      (2,6)*{y},
      (2,2)*{x},
      {(0,-2) \ar @{{*}-{*}} (0,-6)},
      {(0,-6) \ar @{{*}-{*}} (0,-10)},
      (2,-2)*{x},
      (2,-6)*{y},
      (2,-10)*{x},
    \end{xy}\right).
 \end{equation*}
 This can easily be checked using theorem~\ref{thm:AZV_flaten} and equation~\eqref{eq:Kontsevich} relating the two representations of \MZVs{}.
\end{Eg}


\subsection{Arborified Hoffman's relations}

We start by showing that the map $\phi$ of definition~\ref{defi:error_tree} is a bijection. First, recall that $\PTN$ is the set of planar rooted trees decorated by $\N^*$, i.e. the canonical basis of $\calPTN$. 
Define the pairing
\begin{equation*}
    \langle T,T'\rangle = \begin{cases}
        & 1\text{ if }T=T' \\
        & 0\text{ if }T\neq T',
    \end{cases}
\end{equation*} 
for any $(T,T')\in\PTN$, and extend it by linearity to a pairing on $\calPTN$.
 
\begin{defi}
    Let $\Sigma$ be an element in $\calPTN$. We define the {\it support} $\mathrm{Supp}(\Sigma)$ of $\Sigma$ to be 
    \begin{equation*}
        \mathrm{Supp}(\Sigma)\coloneqq\{T\in\PTN \,|\, \langle\Sigma,T\rangle\neq0\}.
    \end{equation*} 
\end{defi}

 Our proof of the bijectivity of $\phi$ relies on the fact that the linearity of a tree (definition~\ref{defi:lin}) decreases when one takes the error term of a tree (definition~\ref{defi:error_tree}).
  \begin{Lemme} \label{lem:lin_error_term}
   Let $T$ be any non-ladder tree in $\PTN$. Then for any tree $T'\in\Supp(T^e)$, we have $\lin(T')<\lin(T)$.
  \end{Lemme}
  \begin{proof}
   Let us write $T=B_+^{n_1}\circ\cdots\circ B_+^{n_p}(B_+^{n_a}(F_a)B_+^{n_b}(F_b)F)$ as in definition~\ref{defi:error_tree}. For any ${T'\in\Supp(T^e)}$ we have either $T'=f_3(T)$ or $T'=f_i(\tilde T)$ for $i\in\{1,2\}$ with $\tilde T$ a tree differing by $T$ only by its decorations. The result follows from the facts that $\lin(T)$ does not depend on the decorations of $T$ and that for any $i\in\IEM{1}{3}, \lin(f_i(T))<\lin(T)$ by definition of $\lin$. 
  \end{proof}

This construction allows us to prove the next important result.
\begin{Prop} \label{Prop:phi_bij}
  The map $\phi : \calPTN\longrightarrow\calPTN$ is a bijection.
\end{Prop}
\begin{proof}
For any $n\in\N, k\in\N$, let us define:
\begin{equation*}
	\PTN^n\coloneqq \left\{t\in\PTN \,\middle|\, \lin(t)=n\right\},\quad \calPTN=\bigoplus_{n=0}^{\infty} \K\PTN^n,\quad\calPTN^{n\leq k}\coloneqq\bigoplus_{n=0}^k \K\PTN^n.
\end{equation*} 
By linearity of $\phi$, it is enough to show
that $\phi$ is a bijection on $\calPTN^{n\leq k}$ for all $k\in \N$. We proceed by induction over $k$.
\begin{description}
	\item[Initialisation:] for $k=0$ the result is obvious as $\phi$ is the identity for ladder trees. 
	\item[Heredity:] let us suppose there exists $k\in\N$ such that $\phi$ is a bijection over $\calPTN^{n\leq k}$. Let us show it is a bijection over $\calPTN^{n \leq k+1}$. For this purpose, one just needs to show for any $T\in\PTN^{k+1}$ has a unique preimage by $\phi$. Let $T$ be any tree in $\PTN^{k+1}$. Then, by lemma~\ref{lem:lin_error_term}:
	\[
	\phi(T)= T + T' \text{ where } T'\in\calPTN^{n\leq k}.
	\]
	By hypothesis, $T'$ has a unique preimage $S'$ by $\phi$. Then by linearity of $\phi$
	$\phi(T-S')=T.$
	So we conclude that $T\in\ima(\phi)$. Moreover, let $S''$ be any preimage of $T$. Then 
	\begin{equation*}
	 \phi(T)=T+T'=\phi(S'')+T'\Longrightarrow T'=\phi(T-S'').
	\end{equation*}
	By unicity of the preimage of $T'$, this implies $T-S''=S'$ if and only if $ S''=T-S'$. Thus the preimage of $T$ is unique and $\phi$ is a bijection over $\calPTN^{n\leq k+1}$.
%
	\end{description}
	As a consequence, $\phi$ is an isomorphism.
		\end{proof}
We can now prove the branched version of Hoffman's relations after a few technical lemmas.
\begin{Lemme} \label{lem:im_spt_y}
 For any tree $T\in\calPTN$, $\fraksPT(T)\shuffle^T\tdun{y}$ lies in $\ima(\fraksPT)$.
\end{Lemme}
\begin{proof}
 Recall that for any tree $T\in\calPTxy$, a vertex $v\in V(T)$ is called a branching vertex if it has at least two direct descendants. Let $\calPTxy^{\rm strict}$ be the subvector space of $\calPTxy$ generated by trees whose leaves and branching vertices are decorated by $y$. Then one easily show (see~\cite[Lemma A.3]{clavier2020double}) that $\fraks^{PN}$ is a bijection between $\calPTN$ and $\calPTxy^{\rm strict}$. Thus, by proposition \ref{Prop:phi_bij}, $\fraks^{PT}=\fraks^{PN}\circ\phi$ is also a bijection between $\calPTN$ and $\calPTxy^{\rm strict}$.
 
 Furthermore, by definition of $\shuffle^T$, we can write $V(t'')=V(t)\sqcup V(t')$ for any tree $t''$ in $\mathrm{Supp}(t\shuffle^T t')$. Then the comb representation of $\shuffle^T$ (theorem~\ref{thm:comb}
 ) implies that each internal (i.e. non leaf) vertex of $t''$ has the same number of direct descendant it had in $t$ or $t'$. Thus $\calPTxy^{\rm strict}$ is a subalgebra of $\left( \calPTxy, \shuffle^T, \emptyset\right)$.
 
 Since $\tdun{y}=\fraks^{PT}(\tdun{1})\in\calPTxy^{\rm strict}$, we have $t\shuffle^T\tdun{y}\in\calPTxy^{\rm strict}$ for any tree $t\in\calPTxy^{\rm strict}$. Then the lemma follows from the fact
 that $\fraks^{PT}$ is a bijection between $\calPTN$ and $\calPTxy^{\rm strict}$.
\end{proof}


We already proved in the previous lemma that $\fraks^{PT}:\calPTN\longrightarrow\calPTxy^{\rm strict}$ is a bijection, thus we have a map $(\fraks^{PT})^{-1}:\calPTxy^{\rm strict}\longrightarrow\calPTN$. Then we have the following technical result.
\begin{Lemme}\label{Lem:Hoffman_step}
     For any tree in $T\in\calPTxy^{\rm strict}$, 
we have
     \begin{equation} \label{eq:trivial_change}
      \flaten_1\circ(\fraks^{PT})^{-1}(T) = \fraks^{-1}\circ\flaten_0(T).
     \end{equation}
    \end{Lemme}
    \begin{proof}
For any $T\in\calPTxy^{\rm strict}=\ima(\fraks^{PT})$, let $T'\in\calPTN$ be its preimage: $T'=(\fraks^{PT})^{-1}(T)$. The commutativity of diagram~\eqref{eq:Ku-Yu_main_result} gives
     \begin{equation*}
      \fraks\circ\flaten_1(T') = \flaten_0\circ\fraks^{PT}(T')=\flaten_0(T)
     \end{equation*}
     where the last equality is obtained by definition of $T'$. Applying $\fraks^{-1}$ to both sides and using again that $T'=(\fraks^{PT})^{-1}(T)$ gives equation \eqref{eq:trivial_change}.
    \end{proof}    
We need one last technical result, whose statement makes sense due to lemma \ref{lem:im_spt_y}.
\begin{Lemme}\label{Lem:Hoffman_conv}
 For any convergent tree $T\in\calPTNconv$, 
 \begin{equation*}
  T\cshuffle^T \tdun{1} - (\fraksPT)^{-1}(\fraksPT(T)\shuffle^T\tdun{y})\in\calPTNconv.
 \end{equation*}
\end{Lemme}
\begin{proof}
By theorem~\ref{thm:comb} for $\cshuffle^T$, for any ${T\in\PTN}$ convergent, we have $T\cshuffle^T \tdun{1}=B_+^1(T)+C$ with $C\in\calPTNconv$. Furthermore, we have by construction $\fraks^{PN}(\calPTNconv)\subseteq\calPTxyconv$. Recall also that, for any non-ladder tree $T\in\PTN$, every vertex of its
process tree $\pr(T)$ (definition~\ref{defi:process_tree}) is decorated by a tree whose root has the same decoration as the root of $T$. Therefore, if $T\in\calPTNconv$ we have $\pr(T)\in\calPT_{\calPTNconv}$. Moreover, note that for any tree $T$, the roots of the trees in the support of its error term $T^e$ (definition~\ref{defi:error_tree}) have the same decoration than the root of $T$. Thus if $T\in\calPTNconv$, we have $T^e\in\calPTNconv$. Therefore we also have $\phi(\calPTNconv)\subseteq\calPTNconv$.

We have shown that $\fraksPT(\calPTNconv)\subseteq\calPTxyconv$. So, by definition of $\shuffle^T$ for any $T\in\calPTNconv$ we obtain $\fraksPT(T)\shuffle^T\tdun{y}=B_+^y(\fraksPT(T))+C'$ with $C'\in\calPTxyconv$.
By lemma \ref{lem:pr_B_plus} we have $\fraksPT(B_+^1(T))=B_+^y(\fraksPT(T))$, and so $(\fraksPT)^{-1}(\fraksPT(T)\shuffle^T\tdun{y})= B_+^1(T)+C''$ with $C''\in\calPTNconv$. Thus as claimed
\begin{equation*}
  T\cshuffle^T \tdun{1} - (\fraksPT)^{-1}(\fraksPT(T)\shuffle^T\tdun{y})=C-C''\in \calPTNconv \qedhere
 \end{equation*}

%
%
\end{proof}

We can now prove the arborified version of Hoffman's regularisation relation~\eqref{eq:Hoffman}.
\begin{thm} \label{thm:arbo_Hoffman}
 For any convergent tree $T\in\calPTNconv$,
 \begin{equation*}
  T\cshuffle^T \tdun{1} - (\fraksPT)^{-1}(\fraksPT(T)\shuffle^T\tdun{y})\in{\rm Ker}(\zetaT).
 \end{equation*}
\end{thm}

\begin{proof}
  Thanks to the previous lemma~\ref{Lem:Hoffman_conv}, for any convergent trees $T\in\calPTNconv$, we can evaluate $\zeta^T$ on $T\cshuffle^T \tdun{1} - (\fraksPT)^{-1}(\fraksPT(T)\shuffle^T\tdun{y})$. Doing so we have (in the computation below, $(\alpha)$ is the word containing only the letter $\alpha$ 
  ):

 \begin{align*}
  & \zetaT\left(T\cshuffle^T \tdun{1} - (\fraksPT)^{-1}(\fraksPT(T)\shuffle^T\tdun{y})\right) \\
  = & \zeta\left(\flaten_1\left(T\cshuffle^T \tdun{1} - (\fraksPT)^{-1}(\fraksPT(T)\shuffle^T\tdun{y})\right)\right) \quad\text{by theorem }\ref{thm:AZV_flaten}\\
  = & \zeta\left(\flaten_1(T\cshuffle^T \tdun{1}) - \flaten_1\circ (\fraksPT)^{-1}(\fraksPT(T)\shuffle^T\tdun{y})\right)\quad\text{by linearity of }\flaten_1 \\
  = & \zeta\left(\flaten_1(T\cshuffle^T \tdun{1}) - \fraks^{-1}\left(\flaten_0(\fraksPT(T)\shuffle^T \tdun{y})\right)\right) \quad\text{by equation }\eqref{eq:trivial_change} \\
  = & \zeta\left(\flaten_1(T)\cshuffle (1) - \fraks^{-1}\left(\flaten_0(\fraksPT(T))\shuffle(y)\right)\right) 
  \quad\text{by lemma }\ref{lemma:flat_dend_morph}\\
  = & \zeta\left(\flaten_1(T)\cshuffle (1) - \fraks^{-1}\left(\fraks(\flaten_1(T))\shuffle(y)\right)\right) \quad\text{by theorem }\ref{thm F.} \\
  = & 0 \quad \text{ by Hoffman's relation \eqref{eq:Hoffman} } \qedhere
 \end{align*}
\end{proof}
Thus the last of the properties of \MZVs{} that (conjecturally) generates every rational relation among them has been lifted to rooted trees. This fulfils the research project that was started by Manchon~\cite{Manchon_16}.

\begin{Egs}
Let us give an example of a relation between \AZVs{} given by this theorem.
\begingroup%
\allowdisplaybreaks%
  \begin{align*}
    &
    \begin{xy}
      {(-4,2) \ar @{{*}-{*}} (0,-2)},
      {(4,2) \ar @{{*}-{*}} (0,-2)},
      (-1,2)*+{1},
      (7,2)*+{1},
      (3,-2)*+{2}
    \end{xy} \cshuffle^T
    \begin{xy}
      {(0,0) \ar @{{*}-{*}} (0,0)},
      (3,0)*+{1}
    \end{xy} - (\fraksPT)^{-1}\left(\fraksPT\left(
    \begin{xy}
      {(-4,2) \ar @{{*}-{*}} (0,-2)},
      {(4,2) \ar @{{*}-{*}} (0,-2)},
      (-1,2)*+{1},
      (7,2)*+{1},
      (3,-2)*+{2}
    \end{xy} \right)\shuffle^T
    \begin{xy}
      {(0,0) \ar @{{*}-{*}} (0,0)},
      (3,0)*+{y}
    \end{xy} \right)\\
    = & 
    \begin{xy}
      {(-4,2) \ar @{{*}-{*}} (0,-2)},
      {(4,2) \ar @{{*}-{*}} (0,-2)},
      (-1,2)*+{1},
      (7,2)*+{1},
      (3,-2)*+{2}
    \end{xy} \cshuffle^T
    \begin{xy}
      {(0,0) \ar @{{*}-{*}} (0,0)},
      (3,0)*+{1}
    \end{xy} - (\fraksPT)^{-1}\left( \left(
    \begin{xy}
      {(0,0) \ar @{{*}-{*}} (0,-4)},
      {(-4,4) \ar @{{*}-{*}} (0,0)},
      {(4,4) \ar @{{*}-{*}} (0,0)},
      (3,-4)*+{x},
      (-1,4)*+{y},
      (7,4)*+{y},
      (3,0)*+{y}
    \end{xy} +
    \begin{xy}
      {(0,-6) \ar @{{*}-{*}} (0,-2)},
      {(0,-2) \ar @{{*}-{*}} (0,2)},
      {(0,2) \ar @{{*}-{*}} (0,6)},
      (3,-6)*+{x},
      (3,-2)*+{y},
      (3,2)*+{x},
      (3,6)*+{y}
    \end{xy} \right) \shuffle^T
    \begin{xy}
      {(0,0) \ar @{{*}-{*}} (0,0)},
      (3,0)*+{y}
    \end{xy} \right) \\
    = & 
    \begin{xy}
      {(0,0) \ar @{{*}-{*}} (0,-4)},
      {(-4,4) \ar @{{*}-{*}} (0,0)},
      {(4,4) \ar @{{*}-{*}} (0,0)},
      (3,-4)*+{1},
      (-1,4)*+{1},
      (7,4)*+{1},
      (3,0)*+{2}
    \end{xy} +
    \begin{xy}
      {(-4,2) \ar @{{*}-{*}} (0,-2)},
      {(4,2) \ar @{{*}-{*}} (0,-2)},
      (-1,2)*+{1},
      (7,2)*+{1},
      (3,-2)*+{3}
    \end{xy} +
    \begin{xy}
      {(4,0) \ar @{{*}-{*}} (4,4)},
      {(-4,0) \ar @{{*}-{*}} (0,-4)},
      {(4,0) \ar @{{*}-{*}} (0,-4)},
      (-1,0)*+{1},
      (7,0)*+{1},
      (7,4)*+{1},
      (3,-4)*+{2}
    \end{xy} +
    \begin{xy}
      {(-4,2) \ar @{{*}-{*}} (0,-2)},
      {(4,2) \ar @{{*}-{*}} (0,-2)},
      (-1,2)*+{1},
      (7,2)*+{2},
      (3,-2)*+{2}
    \end{xy} +
    \begin{xy}
      {(4,0) \ar @{{*}-{*}} (4,4)},
      {(-4,0) \ar @{{*}-{*}} (0,-4)},
      {(4,0) \ar @{{*}-{*}} (0,-4)},
      (-1,0)*+{1},
      (7,0)*+{1},
      (7,4)*+{1},
      (3,-4)*+{2}
    \end{xy} - (\fraksPT)^{-1}\left(
    \begin{xy}
      {(0,-6) \ar @{{*}-{*}} (0,-2)},
      {(0,2) \ar @{{*}-{*}} (0,-2)},
      {(-4,6) \ar @{{*}-{*}} (0,2)},
      {(4,6) \ar @{{*}-{*}} (0,2)},
      (3,-6)*+{y},
      (3,-2)*+{x},
      (-1,6)*+{y},
      (7,6)*+{y},
      (3,2)*+{y}
    \end{xy} +
    \begin{xy}
      {(0,-8) \ar @{{*}-{*}} (0,-4)},
      {(0,-4) \ar @{{*}-{*}} (0,0)},
      {(0,0) \ar @{{*}-{*}} (0,4)},
      {(0,4) \ar @{{*}-{*}} (0,8)},
      (3,-8)*+{y},
      (3,-4)*+{x},
      (3,0)*+{y},
      (3,4)*+{x},
      (3,8)*+{y}
    \end{xy} \right) + (\fraksPT)^{-1}\left(
    \begin{xy}
      {(0,-8) \ar @{{*}-{*}} (0,-4)},
      {(0,-4) \ar @{{*}-{*}} (0,0)},
      {(0,0) \ar @{{*}-{*}} (0,4)},
      {(0,4) \ar @{{*}-{*}} (0,8)},
      (3,-8)*+{x},
      (3,-4)*+{y},
      (3,0)*+{y},
      (3,4)*+{x},
      (3,8)*+{y}
    \end{xy} \right) \\
    & - (\fraksPT)^{-1}\left(
    \begin{xy}
      {(0,-6) \ar @{{*}-{*}} (0,-2)},
      {(0,2) \ar @{{*}-{*}} (0,-2)},
      {(-4,6) \ar @{{*}-{*}} (0,2)},
      {(4,6) \ar @{{*}-{*}} (0,2)},
      (3,-6)*+{x},
      (3,-2)*+{y},
      (-1,6)*+{y},
      (7,6)*+{y},
      (3,2)*+{y}
    \end{xy} +
    \begin{xy}
      {(0,-8) \ar @{{*}-{*}} (0,-4)},
      {(0,-4) \ar @{{*}-{*}} (0,0)},
      {(0,0) \ar @{{*}-{*}} (0,4)},
      {(0,4) \ar @{{*}-{*}} (0,8)},
      (3,-8)*+{x},
      (3,-4)*+{y},
      (3,0)*+{y},
      (3,4)*+{x},
      (3,8)*+{y}
    \end{xy} \right) - (\fraksPT)^{-1}\left(
    \begin{xy}
      {(0,-6) \ar @{{*}-{*}} (0,-2)},
      {(4,2) \ar @{{*}-{*}} (4,6)},
      {(-4,2) \ar @{{*}-{*}} (0,-2)},
      {(4,2) \ar @{{*}-{*}} (0,-2)},
      (3,-6)*+{x},
      (3,-2)*+{y},
      (-1,2)*+{y},
      (7,2)*+{y},
      (7,6)*+{y}
    \end{xy} +
    \begin{xy}
      {(0,-8) \ar @{{*}-{*}} (0,-4)},
      {(0,-4) \ar @{{*}-{*}} (0,0)},
      {(0,0) \ar @{{*}-{*}} (0,4)},
      {(0,4) \ar @{{*}-{*}} (0,8)},
      (3,-8)*+{x},
      (3,-4)*+{y},
      (3,0)*+{y},
      (3,4)*+{x},
      (3,8)*+{y}
    \end{xy} +
    \begin{xy}
      {(0,-8) \ar @{{*}-{*}} (0,-4)},
      {(0,-4) \ar @{{*}-{*}} (0,0)},
      {(0,0) \ar @{{*}-{*}} (0,4)},
      {(0,4) \ar @{{*}-{*}} (0,8)},
      (3,-8)*+{x},
      (3,-4)*+{y},
      (3,0)*+{x},
      (3,4)*+{y},
      (3,8)*+{y}
    \end{xy} \right) - (\fraksPT)^{-1}\left(
    \begin{xy}
      {(0,-6) \ar @{{*}-{*}} (0,-2)},
      {(4,2) \ar @{{*}-{*}} (4,6)},
      {(-4,2) \ar @{{*}-{*}} (0,-2)},
      {(4,2) \ar @{{*}-{*}} (0,-2)},
      (3,-6)*+{x},
      (3,-2)*+{y},
      (-1,2)*+{y},
      (7,2)*+{y},
      (7,6)*+{y}
    \end{xy} +
    \begin{xy}
      {(0,-8) \ar @{{*}-{*}} (0,-4)},
      {(0,-4) \ar @{{*}-{*}} (0,0)},
      {(0,0) \ar @{{*}-{*}} (0,4)},
      {(0,4) \ar @{{*}-{*}} (0,8)},
      (3,-8)*+{x},
      (3,-4)*+{y},
      (3,0)*+{y},
      (3,4)*+{x},
      (3,8)*+{y}
    \end{xy} +
    \begin{xy}
      {(0,-8) \ar @{{*}-{*}} (0,-4)},
      {(0,-4) \ar @{{*}-{*}} (0,0)},
      {(0,0) \ar @{{*}-{*}} (0,4)},
      {(0,4) \ar @{{*}-{*}} (0,8)},
      (3,-8)*+{x},
      (3,-4)*+{y},
      (3,0)*+{x},
      (3,4)*+{y},
      (3,8)*+{y}
    \end{xy} \right)\\
    = & 
    \begin{xy}
      {(0,0) \ar @{{*}-{*}} (0,-4)},
      {(-4,4) \ar @{{*}-{*}} (0,0)},
      {(4,4) \ar @{{*}-{*}} (0,0)},
      (3,-4)*+{1},
      (-1,4)*+{1},
      (7,4)*+{1},
      (3,0)*+{2}
    \end{xy} +
    \begin{xy}
      {(-4,2) \ar @{{*}-{*}} (0,-2)},
      {(4,2) \ar @{{*}-{*}} (0,-2)},
      (-1,2)*+{1},
      (7,2)*+{1},
      (3,-2)*+{3}
    \end{xy} +
    \begin{xy}
      {(4,0) \ar @{{*}-{*}} (4,4)},
      {(-4,0) \ar @{{*}-{*}} (0,-4)},
      {(4,0) \ar @{{*}-{*}} (0,-4)},
      (-1,0)*+{1},
      (7,0)*+{1},
      (7,4)*+{1},
      (3,-4)*+{2}
    \end{xy} +
    \begin{xy}
      {(-4,2) \ar @{{*}-{*}} (0,-2)},
      {(4,2) \ar @{{*}-{*}} (0,-2)},
      (-1,2)*+{1},
      (7,2)*+{2},
      (3,-2)*+{2}
    \end{xy} +
    \begin{xy}
      {(4,0) \ar @{{*}-{*}} (4,4)},
      {(-4,0) \ar @{{*}-{*}} (0,-4)},
      {(4,0) \ar @{{*}-{*}} (0,-4)},
      (-1,0)*+{1},
      (7,0)*+{1},
      (7,4)*+{1},
      (3,-4)*+{2}
    \end{xy} - \left(
    \begin{xy}
      {(0,0) \ar @{{*}-{*}} (0,-4)},
      {(-4,4) \ar @{{*}-{*}} (0,0)},
      {(4,4) \ar @{{*}-{*}} (0,0)},
      (3,-4)*+{1},
      (-1,4)*+{1},
      (7,4)*+{1},
      (3,0)*+{2}
    \end{xy} +
    \begin{xy}
      {(0,0) \ar @{{*}-{*}} (0,-4)},
      {(-4,4) \ar @{{*}-{*}} (0,0)},
      {(4,4) \ar @{{*}-{*}} (0,0)},
      (3,-4)*+{2},
      (-1,4)*+{1},
      (7,4)*+{1},
      (3,0)*+{1}
    \end{xy} +
    \begin{xy}
      {(4,0) \ar @{{*}-{*}} (4,4)},
      {(-4,0) \ar @{{*}-{*}} (0,-4)},
      {(4,0) \ar @{{*}-{*}} (0,-4)},
      (-1,0)*+{1},
      (7,0)*+{1},
      (7,4)*+{1},
      (3,-4)*+{2}
    \end{xy} +
    \begin{xy}
      {(4,0) \ar @{{*}-{*}} (4,4)},
      {(-4,0) \ar @{{*}-{*}} (0,-4)},
      {(4,0) \ar @{{*}-{*}} (0,-4)},
      (-1,0)*+{1},
      (7,0)*+{1},
      (7,4)*+{1},
      (3,-4)*+{2}
    \end{xy} \right)\\
    & +
    \begin{xy}
      {(0,-4) \ar @{{*}-{*}} (0,0)},
      {(0,0) \ar @{{*}-{*}} (0,4)},
      (3,-4)*+{2},
      (3,0)*+{1},
      (3,4)*+{2}
    \end{xy} =
    \begin{xy}
      {(-4,2) \ar @{{*}-{*}} (0,-2)},
      {(4,2) \ar @{{*}-{*}} (0,-2)},
      (-1,2)*+{1},
      (7,2)*+{1},
      (3,-2)*+{3}
    \end{xy} +
    \begin{xy}
      {(-4,2) \ar @{{*}-{*}} (0,-2)},
      {(4,2) \ar @{{*}-{*}} (0,-2)},
      (-1,2)*+{1},
      (7,2)*+{2},
      (3,-2)*+{2}
    \end{xy} -
    \begin{xy}
      {(0,0) \ar @{{*}-{*}} (0,-4)},
      {(-4,4) \ar @{{*}-{*}} (0,0)},
      {(4,4) \ar @{{*}-{*}} (0,0)},
      (3,-4)*+{2},
      (-1,4)*+{1},
      (7,4)*+{1},
      (3,0)*+{1}
    \end{xy} +
    \begin{xy}
      {(0,-4) \ar @{{*}-{*}} (0,0)},
      {(0,0) \ar @{{*}-{*}} (0,4)},
      (3,-4)*+{2},
      (3,0)*+{1},
      (3,4)*+{2}
    \end{xy}
  \end{align*}
  \endgroup%
Then the arborified version of Hoffman’s regularisation relation (theorem~\ref{thm:arbo_Hoffman}) gives
\[
\zeta^T\left(
\begin{xy}
	{(-4,2) \ar @{{*}-{*}} (0,-2)},
	{(4,2) \ar @{{*}-{*}} (0,-2)},
	(-1,2)*+{1},
	(7,2)*+{1},
	(3,-2)*+{3}
\end{xy} +
\begin{xy}
	{(-4,2) \ar @{{*}-{*}} (0,-2)},
	{(4,2) \ar @{{*}-{*}} (0,-2)},
	(-1,2)*+{1},
	(7,2)*+{2},
	(3,-2)*+{2}
\end{xy} -
\begin{xy}
	{(0,0) \ar @{{*}-{*}} (0,-4)},
	{(-4,4) \ar @{{*}-{*}} (0,0)},
	{(4,4) \ar @{{*}-{*}} (0,0)},
	(3,-4)*+{2},
	(-1,4)*+{1},
	(7,4)*+{1},
	(3,0)*+{1}
\end{xy} +
\begin{xy}
	{(0,-4) \ar @{{*}-{*}} (0,0)},
	{(0,0) \ar @{{*}-{*}} (0,4)},
	(3,-4)*+{2},
	(3,0)*+{1},
	(3,4)*+{2}
\end{xy}\right) = 0.
\]

Using theorem~\ref{thm:AZV_flaten} this implies
\[
2\zeta(3,1,1)+\zeta(3,2)+\zeta(2,1,2)+\zeta(2,2,1)+\zeta(2,3)-2\zeta(2,1,1,1)= 0.
\]
which one can straightforwardly check to be true, using for example the standard (conjectured) basis of \MZVs{} of weight 5 given by $\{\zeta(3,2),\zeta(2,3)\}$.
Although the arborified version of Hoffman’s regularisation relation is derived from Hoffman’s regularisation relation, applying Hoffman’s regularisation relation directly to the word $w=(2,1,1)$, namely considering
\[
(2,1,1)\cshuffle (1) - (\fraks)^{-1}(\fraks(2,1,1)\shuffle (y))
\]
gives
\[
\zeta(3,1,1)+\zeta(2,1,2)+\zeta(2,2,1)-\zeta(2,1,1,1)= 0.
\]
Thus the equivalence between relations obtain from Hoffman's regularisation relation and the ones obtained from the arborified version is not immediate.
\end{Egs}

\section{From Arborified to Multiple Zeta Values} \label{sec:relating_zetas}

\subsection{A conjecture regarding \AZVs}

Recall that a planar rooted tree can be defined as a non-planar rooted tree together with the total order $O_v$ on the set of direct descendants of each vertex $v$. Assuming that a tree is an oriented graph whose edges are oriented towards its root, we can write rigorously:
\begin{gather*}
    \calPT\ni(T,\{O_v \, | \, v\in V(T)\}) 
\end{gather*}
where $T$ is a non-planar rooted tree and $O_v$ is a total order on $\{v'\in V(T) \, | \, a(v')=v\}$, where $a$ is the map that maps any vertex to its immediate predecessor.
In this definition, edges are oriented towards the leaves of the tree.

Let $\Omega$ be any set and denote $\calT_{\Omega}$ the set of rooted trees decorated by $\Omega$. The construction described above easily generalises to $\Omega$-decorated rooted trees. Then forgetting the orientation set $\{O_v\, |\,  v\in V(T)\}$ we obtain a map 
\[
F_\Omega:\left\lbrace\begin{array}{rcl}
	\calPTO&\longrightarrow &\calT_\Omega \\
	(T,\{O_v\, |\, v\in V(T)\})& \longmapsto & T.
\end{array}\right.
\]
From the definitions of \AZVs{}, we directly have that if two convergent trees have the same image in $\calT_\Omega$ under the map $F_\Omega$, then their associated \AZVs{} are the same. This holds for $\Omega=\N^*$ as well as $\Omega=\{x,y\}$.

\begin{Prop} \label{prop:planarity} 
    Let $(S,T)\in\calPTNconv \times \calPTNconv$ (resp. $(S,T)\in\calPTxyconv \times \calPTxyconv$). Then
    \begin{equation*}
        F_{\N^*}(S)=F_{\N^*}(T)~\Longrightarrow~\zeta^T(S)=\zeta^T(T)\quad\left(\text{\normalfont{resp. } }F_{\{x,y\}}(S)=F_{\{x,y\}}(T)~\Longrightarrow~\zeta_\shuffle^T(S)=\zeta_\shuffle^T(T)~\right).
    \end{equation*}
    We call these relations the \emph{planarity relations}.
\end{Prop}
\begin{proof}
 Direct from definition~\ref{def:azvs}: $\zeta^T(t)$ and $\zeta^T_\shuffle(s)$ do not depend on the sets $O_v$ of $t$ and $s$.
\end{proof}
We make the following conjecture, which is a natural generalisation to planar rooted trees and \AZVs{} of conjecture~\ref{conj:mzvs} for words and \MZVs{}.
\begin{conj} \label{conj:azvs}
 Any rational relation among \AZVs{} comes from the quasi-shuffle relations, the shuffle relations (theorem \ref{thm:azv_shuffle_stuffle}), arborified Hoffman's relation (theorem \ref{thm:arbo_Hoffman}), and the planarity relations (proposition \ref{prop:planarity}); together with the iterated integral representation of \AZVs{} (corollary \ref{coro:arbo_Konts}).
\end{conj}

\subsection{Relating the conjectures}

In order to relate conjectures~\ref{conj:azvs} and~\ref{conj:mzvs} we need to be more \emph{precise} in their statements. We will work here with the iterated series representation of \MZVs{} and \AZVs{}, but the similar arguments could be used for their iterated integral representation, and these two choices are equivalent.

First, let $R_W$ be the ideal of $\PolW$ generated by the quasi-shuffle and shuffle relations~\eqref{eq:shuffle_stuffle_zeta}, the double shuffle relation~\eqref{eq:Kontsevich}, and Hoffman's regularisation relations~\eqref{eq:Hoffman}. To be more precise, $R_W$ is generated as an ideal by the relations
\begin{equation*}
 w_1w_2-w_1\cshuffle w_2,\quad w_1w_2-\fraks^{-1}\left(\fraks(w_1)\shuffle\fraks(w_2)\right), \quad w\cshuffle(1)-\fraks^{-1}\left(\fraks(w)\shuffle(y)\right)
\end{equation*}
for any convergent words $w$, $w_1$ and $w_2$. In the relations above, the products $w_1w_2$ is the commutative polynomial product in $\PolW$.

Now, we can extend multiplicatively the map $\zeta:\Wcv\longrightarrow\R$ to a map (also denoted $\zeta$ for the sake of readability) $\zeta:\PolW\longrightarrow\R$. Then the known results on \MZVs{} imply that $R_W\subseteq\Ker(\zeta)$. Conjecture~\ref{conj:mzvs} states that the reverse inclusion also holds.


We can formulate conjecture~\ref{conj:azvs} in the same language. Let $R_T$ be the ideal of $\PolT$ generated as an ideal by the relations
\begin{align} \label{eq:def_RT}
\begin{split}
 ST & -S\cshuffle^T T,\quad ST-(\fraksPT)^{-1}\left(\fraksPT(S)\shuffle^T\fraksPT(T)\right),\\
 & U\cshuffle^T\tdun{1}-(\fraksPT)^{-1}\left(\fraksPT(U)\shuffle^T\tdun{y}\right),\quad S_1-S_2 
\end{split}
\end{align}
for any convergent trees $S$, $T$, $U$, and any convergent trees $S_1$ and $S_2$ such that ${F_{\N^*}(S_1)=F_{\N^*}(S_2)}$. Here, the products $ST$ in the first two relations are the commutative polynomial products in $\PolT$.

\begin{Rq}
	The quasi-shuffle product $\cshuffle^T$ is not commutative but the polynomial product in $\PolT$ is, so we have
	\[
	ST - S\cshuffle T \in R_T \text{ and } TS - T\cshuffle S \in R_T\text{ so } \zeta^T(T\cshuffle S - S \cshuffle T) =0. 
	\]
\end{Rq}

Again, extending multiplicatively the map $\zeta^T$ to a map (that we denote with the same symbol) $\zeta^T:\PolT\longrightarrow\R$, we have shown in this paper that $R_T\subseteq \Ker(\zetaT)$. Conjecture~\ref{conj:azvs} then states that this inclusion is an equality.


To relate these two conjectures, let us also multiplicatively extend the map $\flaten_1$ to a map $\Flaten_1:\PolT\longrightarrow\PolW$. It is linear and its action on monomials is given by
\begin{equation*}
 \Flaten_1\left(\prod_{i\in I}(T_i)^{n_i}\right):=\prod_{i\in I}\flaten_1(T_i)^{n_i}
\end{equation*}
where the products are the commutative polynomial products in $\PolT$ and $\PolW$ respectively and $I$ is a finite set The conjectures are linked thanks to the following technical lemma.
\begin{Lemme} \label{lem:flaten_ideals}
 We have $\Flaten_1(R_T)\subseteq R_W$.
\end{Lemme}

\begin{proof}
By linearity and multiplicativity of $\Flaten_1$, it is enough to show that it sends the generating relations~\eqref{eq:def_RT} of $R_T$ to elements of $R_W$. We show that this holds for each family of relations independently. In the following, $S$, $T$, and $U$ are any convergent trees and $S_1$ and $S_2$ are any convergent trees such that $F_{\N^*}(S_1)=F_{\N^*}(S_2)$.
 \begin{itemize}
  \item Since $S\cshuffle^T T\in\calPTNconv$ we have
  \begin{align*}
   \Flaten_1(ST -S\cshuffle^T T) & =\flaten_1(S)\flaten_1(T)-\flaten_1(S\cshuffle^T T) \\
   & = \flaten_1(S)\flaten_1(T)-\flaten_1(S)\cshuffle\flaten_1(T) \\
   &\in R_W.
  \end{align*}
  For the second equality, we have used the tridendriform part of Lemma~\ref{lemma:flat_dend_morph}, i.e. that $\flaten_1$ is an algebra morphism for the quasi-shuffle products.
  
  \item Again, we have $(\fraksPT)^{-1}\left(\fraksPT(S)\shuffle^T\fraksPT(T)\right)\in\calPTNconv$. Therefore we can write
  \begin{align*}
				&\Flaten_1  \left(ST - (\fraksPT)^{-1}\left(\fraksPT(S) \shuffle^T \fraksPT(T) \right)\right) \\
				  =&\flaten_1(S)\flaten_1(T) - \flaten_1\left( (\fraksPT)^{-1}\left(\fraksPT(S) \shuffle^T \fraksPT(T) \right)\right)\\
				=& \flaten_1(S) \flaten_1(T) - \fraks^{-1} \circ \flaten_0 \left( \fraksPT(S) \shuffle^T \fraksPT(T)\right) \tag{by equation~\eqref{eq:trivial_change}} \\
				=& \flaten_1(S)  \flaten_1(T) - \fraks^{-1} \left( \flaten_0 \left(\fraksPT(S)\right) \shuffle \flaten_0  \left(\fraksPT(T)\right)\right) \tag{by lemma~\ref{lemma:flat_dend_morph}}\\
				=& \flaten_1(S)  \flaten_1(T) - \fraks^{-1} \big(  \fraks\left(\flaten_1(S)\right) \shuffle \fraks \left(\flaten_1(T)\right)\big) \\
				\in& R_W.
			\end{align*}
The last equality comes from the lower part of diagram~\eqref{eq:Ku-Yu_main_result}.
  \item Next, we have by lemma~\ref{Lem:Hoffman_conv} that $U\cshuffle^T\tdun{1}-(\fraksPT)^{-1}\left(\fraksPT(U)\shuffle^T\tdun{y}\right)\in\calPTNconv$. Therefore 
  \begin{align*}
   &\Flaten_1  \left(U\cshuffle^T\tdun{1}-(\fraksPT)^{-1}\left(\fraksPT(U)\shuffle^T\tdun{y}\right)\right) \\
    =& \flaten_1\left(U\cshuffle^T \tdun{1} - (\fraksPT)^{-1}(\fraksPT(U)\shuffle^T\tdun{y})\right) \\
    =& \flaten_1(U) \cshuffle (1) - \flaten_1\circ (\fraksPT)^{-1}(\fraksPT(U)\shuffle^T\tdun{y}) \tag{by lemma~\ref{lemma:flat_dend_morph}} \\
    =& \flaten_1(U) \cshuffle (1) - \fraks^{-1}\circ\flaten_0(\fraksPT(U)\shuffle^T\tdun{y}) \tag{by equation~\eqref{eq:trivial_change}} \\
    =& \flaten_1(U) \cshuffle (1) - \fraks^{-1}\circ\left(\flaten_0\circ\fraksPT(U)\shuffle (y)\right) \tag{by lemma~\ref{lemma:flat_dend_morph}} \\
    =& \flaten_1(U) \cshuffle (1) - \fraks^{-1}\circ\left(\fraks\circ\flaten_1(U)\shuffle (y)\right) \in R_W.
  \end{align*}
Once again, the last equality comes from the lower part of diagram~\eqref{eq:Ku-Yu_main_result}.
  \item Finally, if $S_1$ and $S_2$ are such that $F_{\N^*}(S_1)=F_{\N^*}(S_2)$ we have by the commutativity of the quasi-shuffle product of words $\cshuffle$ that $\Flaten_1(S_1)=\Flaten_1(S_2)$. Therefore, the last family of relations in~\eqref{eq:def_RT} are mapped to $0\in R_W$ in $\PolW$ by $\Flaten_1$.
 \end{itemize}
 Thus each relation defining $R_T$ is mapped to a relation defining $R_W$. So ${\Flaten_1(R_T)\subseteq R_W}$.
\end{proof}
We can now prove the main result of this section.
\begin{thm} \label{thm:relating_conj}
    Conjecture~\ref{conj:azvs} implies conjecture~\ref{conj:mzvs}.
\end{thm}

\begin{proof}
 Let us assume that conjecture~\ref{conj:azvs} holds, and let $R\in\PolW$. Assume $R\in\Ker(\zeta)$, we want to prove that we then have $R\in R_W$.
 
 Recall that $\iota:\calW_{\N^*}\longrightarrow\calPTN$ is the canonical embedding of words into planar rooted trees from equation~\eqref{eq:iota_map}. Let us expand $\iota$ as a multiplicative map, written with the same symbol, $\iota:\PolW\longrightarrow\PolT$. Notice that since $\flaten_1\circ\iota={\rm Id}_{\calW_{\N^*}}$ we have $\Flaten_1\circ\iota={\rm Id}_{\PolW}$.
 
 Then since $\iota$, $\zeta$, and $\zeta^T$ have been extended by multiplicativity of $\PolW$ and $\PolT$ we have $\iota(R)\in\Ker(\zeta^T)$. Then conjecture~\ref{conj:azvs} implies $\iota(R)\in R_T$. Applying lemma~\ref{lem:flaten_ideals} gives
 \begin{equation*}
  R_W \ni \Flaten_1\circ\iota(R)=R.
 \end{equation*}
 Therefore, conjecture~\ref{conj:mzvs} is, as claimed, a consequence of  conjecture~\ref{conj:azvs}.
\end{proof}


\section{Conclusion and remaining questions}
	
	Thanks to dendriform and tridendriform structures, we introduced a shuffle product for trees which, combined with the binarisation map from the third author's work, enable us to state conjecture~\ref{conj:azvs} implying the usual one (conjecture~\ref{conj:mzvs}) for \MZVs{}. This answers the question asked by Manchon in 2016~\cite{Manchon_16}. For further studies, it will be interesting to look at:
	\begin{itemize}
		\item is conjecture~\ref{conj:azvs} equivalent to conjecture~\ref{conj:mzvs}~? We proved one implication but the other one  (which we expect to be true as well) requires new tools lying beyond the scope of this work.
		\item is the binarisation map of definition~\ref{defi:error_tree} compatible with the algebraic structures introduced in section~\ref{sec:tridend} ? This is not the case for words but could be for trees.
		\item these dendriform and tridendriform algebras over trees have rigid Hopf algebraic structures. Can these
		provide tools to describe the relations among \AZVs{} ? 
		\item can a motivic version of \AZVs{} could be defined and studied in order to tackle conjecture~\ref{conj:azvs} and therefore conjecture~\ref{conj:mzvs} ?
	\end{itemize} 

\addcontentsline{toc}{section}{References}

	\bibliography{biblio_MZVs_20260901}
	\bibliographystyle{plain}
	
\end{document}